\documentclass[11pt]{amsart}
\usepackage{amssymb}

\usepackage{color}
\usepackage{graphics}
\usepackage{amsmath,amscd}
\usepackage{amsthm}
\usepackage[all,cmtip]{xy}
\usepackage{tikz}
\usepackage{graphicx}
\usepackage{epstopdf}
\usepackage{mathtools}
\usepackage{amssymb}
\usepackage{amsbsy}
\usepackage{bm}
\usepackage{tikz-cd}
\usepackage{enumerate}
\usepackage{enumitem}
\usepackage{float}
\theoremstyle{definition}
\newtheorem{exmp}{Example}[section]
\usepackage[colorlinks=green,linkcolor=blue,citecolor=red,urlcolor=red]{hyperref}

\newcommand{\N}{\mathbb{N}}

\newtheorem{proposition}{Proposition}[section]
\newtheorem{theorem}[proposition]{Theorem}
\newtheorem{definition}[proposition]{Definition}
\newtheorem{lemma}[proposition]{Lemma}

\newtheorem{corollary}[proposition]{Corollary}
\newtheorem{remark}[proposition]{Remark}

\begin{document}
	
	\title{On open books and embeddings of smooth and contact manifolds}
	
	

	\author{Arijit Nath}
	\address{Indian Institute of Technology Madras}
	\email{arijit2357@gmail.com}
	
	\author{Kuldeep Saha}
	\address{Indian Institute of Science Education and Research Bhopal}
	\email{kuldeep.saha@gmail.com}
	
	\begin{abstract}
		  We discuss embedding of manifolds in the category of open books, contact manifolds and contact open books. We prove an open book version of the Haefliger--Hirsch embedding theorem by showing that every $k$-connected closed $n$-manifold ($n\geq 7$, $k < \frac{n-4}{2}$) admits an open book embedding in the trivial open book of $\mathbb{S}^{2n-k}$. We then prove that every closed manifold $M^{2n+1}$ that bounds an achiral Lefschetz fibration, admits open book embedding in the trivial open book of $\mathbb{S}^{2\lfloor\frac{3n}{2}\rfloor + 3}$. We also prove that every closed manifold $M^{2n+1}$ bounding an achiral Lefschetz fibration admits a contact structure that isocontact embeds in the standard contact structure on $\mathbb{R}^{2n+3}.$ Finally, we give various examples of contact open book embeddings of contact $(2n+1)$-manifolds in the trivial supporting open book of the standard contact structure on $\mathbb{S}^{4n+1}.$ 
	\end{abstract}
	
	\maketitle

	\section{Introduction}
	
	Embedding of manifolds is a fundamental problem in topology. In particular, embedding in $\mathbb{R}^N$ or $\mathbb{S}^N$ is a much studied problem. The Whitney embedding theorem says that every $n$-manifold embeds in $\mathbb{S}^{2n}$. The works of Haefliger--Hirsch \cite{HH} and Hirsch \cite{Hi} showed that every orientable $n$-manifold can be embedded in $\mathbb{S}^{2n-1}$. Haefliger--Hirsch \cite{HH} further generalized this to prove that every $k$-connected oriented $n$-manifold embeds in $\mathbb{S}^{2n-k-1}$. After Nash \cite{N} solved the $C^1$-isometric embedding problem, people started focusing on embedding problems with a given geometric structure. In the present article, we will be discussing embedding problems in three categories of geometric/topological structures: open books, contact structures and contact open books.
	
	An open book decomposition on a closed $n$-manifold is obtained by removing a codimension-$2$ closed submanifold and fibering the complement over $\mathbb{S}^1$. In particular, an open book is determined by a fiber $\Sigma^{n-1}$ (called \emph{page}) and the return map $\phi$ (called \emph{monodromy}) of this fibration. We denote such an open book by $\mathcal{O}b(\Sigma,\phi)$. An \emph{open book embedding} of $M_1 = \mathcal{O}b(\Sigma_1,\phi_1)$ in $M_2 = \mathcal{O}b(\Sigma_2,\phi_2)$ is an embedding $f$ of $M_1$ in $M_2$ such that the restriction of $\mathcal{O}b(\Sigma_2,\phi_2)$ on $f(M_1)$ induces the open book $\mathcal{O}b(\Sigma_1,\phi_1)$. In the present article we will be considering the case : $M_2 = \mathbb{S}^{N+1} = \mathcal{O}b(\mathbb{D}^{N}, id)$ for some $N$. This is called the trivial open book of $\mathbb{S}^{N+1}$. Clearly, an open book embedding in the trivial open book gives a stronger form of embedding than just smooth embedding.
	
	Combining the construction of open book decomposition due to Winkelnkemper \cite{Wi} and embedding and isotopy results of Haefliger--Hirsch \cite{HH}, we get the following open book version of the Haefliger--Hirsch embedding theorem.
	
	\begin{theorem}\label{open book thm 1}
		Let $n \geq 7$ and $0 \leq k < \frac{1}{2}(n-4)$. Then, every $k$-connected closed manifold $M^n$ admits an open book embedding in $\mathbb{S}^{2n-k} = \mathcal{O}b(\mathbb{D}^{2n-k-1},id)$.
	\end{theorem}

   Next we prove a smooth open book embedding theorem for manifolds bounding an \emph{achiral Lefschetz fibration}. An achiral Lefschetz fibration is just like a Lefschetz fibration over a disk $D^2$ with Weinstein fibers, where the monodromy can be a composition of both positive and negative Dehn--Seidel twists. For precise definitions see Section \ref{weinstein}.
   
   \begin{theorem} \label{top sein embed}
   	Let $M^{2n+1}$ be a closed manifold that bounds an achiral Lefschetz fibration. Then, $M^{2n+1}$ admits open book embedding in $\mathcal{O}b(\mathbb{D}^{2\lfloor \frac{3n}{2} \rfloor + 1},id) = \mathbb{S}^{2\lfloor\frac{3n}{2}\rfloor + 3}$.
   \end{theorem}

   In fact, every smooth $(2n+1)$-manifold $M^{2n+1}$ bounding an achiral Lefschetz fibration can be embedded in a much lower codimension in both the smooth and contact category. For the definition of isocontact embedding see Section \ref{isocontact}.

   \begin{theorem} \label{codim 2 embedding}
   	
   	Let $M^{2n+1}$ be a closed oriented manifold that bounds an achiral Lefschetz fibration.
   	
   	\begin{enumerate}
   		
   		\item There exists a contact structure $\xi_0$ on $M^{2n+1}$ such that $(M^{2n+1},\xi_0)$ isocontact embeds in $(\mathbb{R}^{2n+3},\xi_{std}).$
   		
   		\item If $M^{2n+1}$ is $(n-1)$-connected, then for $n \equiv 4 \pmod 8$, every contact structure $(M^{2n+1},\xi)$ isocontact embeds in $(\mathbb{R}^{2n+3},\xi_{std}).$ 
   		
   	\end{enumerate}

   \end{theorem} 
   
   Giroux and Pardon \cite{GP} has proved that every Weinstein or Stein manifold admits an achiral Lefschetz fibration structure. Therefore, if $M^{2n+1}$ does not admit embedding in $\mathbb{R}^{2n+3}$, then $M^{2n+1}$ can not admit any Stein or Weinstein filling. For example, spaces like $\mathbb{R}P^{2n+1}$, $S^{2m+1} \times \mathbb{C}P^n$, for $n$ large, can not admit Stein or Weinstein filling. Moreover, given such an $M^{2n+1}$, if $\mathcal{O}b(\Sigma_M,\phi_M)$ is an open book supporting a contact structure $\xi_M$ on $M$. Then, $\phi_M$ cannot be written as a product of Dehn--Seidel twists. 

\

We also prove some \emph{contact open book embedding} results. The trivial contact open book of $\mathbb{S}^{2N+1}$, supporting the standard contact structure $\xi_{std}$, has page the standard symplectic disk $(\mathbb{D}^{2N},d\lambda_0)$ and monodromy identity, and is denoted by $\mathcal{O}b(\mathbb{D}^{2N},d\lambda_0,id)$. By the $h$-principles of Gromov \cite{Gr}, every contact $(2n+1)$-manifold contact open book embeds in $\mathcal{O}b(\mathbb{D}^{4n+2},d\lambda_0,id)$. We give some examples of contact open book embeddings of simply connected contact $(2n+1)$-manifolds in the trivial contact open book of $(\mathbb{S}^{4n+1},\xi_{std})$. Note that Kasuya \cite{Ka} has already proved that every $2$-connected contact $(2n+1)$-manifold isocontact embeds in $(\mathbb{S}^{4n+1},\xi_{std})$. First we show the following.

\begin{theorem}\label{contact open book thm}
	Let $(V^{2n},\partial V^{2n},d\lambda_V)$ be a simply connected exact symplectic manifold. If $(V^{2n},\partial V^{2n}$\\$,d\lambda_V)$ has a proper isosymplectic embedding in $(\mathbb{D}^{4n},\partial \mathbb{D}^{4n},d\lambda_0)$, then any open book $\mathcal{O}b(V^{2n},d\lambda_V,\phi)$ contact open book embeds in $\mathcal{O}b(\mathbb{D}^{4n},d\lambda_0,id)$ for $n \geq 2$.

\end{theorem}

\noindent As an application of Theorem \ref{contact open book thm}, we find various classes of contact $(2n+1)$-manifolds that admit contact open book embedding in the trivial open book of $(\mathbb{S}^{4n+1},\xi_{std})$. For instance, we show the following.

\begin{corollary} \label{corollary1}
	Let $(W^{2n},d\lambda_W)$ be a simply connected sub-critical Weinstein domain. Then $\mathcal{O}b(W^{2n},d\lambda_W,\phi_W)$ contact open book embeds in $\mathcal{O}b(\mathbb{D}^{4n},d\lambda_0,id)$, for $n \geq 3$.
\end{corollary}

\noindent Another such class of contact manifolds is given by the following. 

\begin{corollary} \label{corollary2}
	
	Let $(W^{2n},d\lambda_W)$ be a simply connected Weinstein domain which has the homotopy type of a CW complex that has no nonzero even dimensional cells. Then $\mathcal{O}b(W^{2n},d\lambda_W,\phi_W)$ contact open book embeds in $\mathcal{O}b(\mathbb{D}^{4n},d\lambda_0,id)$, for $n \geq 3$.
	
\end{corollary}

\noindent For example, let $n = 2m+1$ and say $(W^{2n},d\lambda_W)$ is obtained by plumbing odd dimensional unit cotangent bundles $DT^*\mathbb{S}^{2m+1}$ with their canonical symplectic forms. Then, for any relative symplectomorphism $\phi$ of $(W,d\lambda_0)$, $\mathcal{O}b(W^{2n},d\lambda_W,\phi)$ contact open book embeds in $\mathcal{O}b(\mathbb{D}^{4n},d\lambda_0, id)$.

 In recent times, some explicit constructions of smooth and contact open book embedding were given in \cite{EF},\cite{EL},\cite{PPS}, \cite{S} and \cite{S2}. While much is known about contact open book embeddings of $3$-manifolds, similar problems for higher dimensional contact manifolds are yet to be studied in full generality.

 \subsection{Acknowledgment} We thank Dishant M. Pancholi for his encouragement and for suggesting the problem of open book embedding. We also thank John B. Etnyre for many helpful comments and suggestions. The first author is supported by the CSIR, India (Fellowship Ref. no. 09/084(0688)/2016-EMR-I).

  \section{preliminaries}
   
  \subsection{Open books}\label{open book def}
   An open book is a decomposition of a manifold into a trivial $2$-disk bundle over a co-dimension $2$ submanifold and a mapping torus.
   
   \begin{definition}[Open book decomposition]\label{def: op book}
   
   An open book decomposition of a closed oriented  manifold $M$ consists of a co-dimension $2$ oriented submanifold $B$ with a trivial normal bundle in $M$ and a fibration $ \pi: M \setminus B \rightarrow \mathbb{S}^1$ such that $\pi^{-1}(\theta)$ is an interior of a co-dimension $1$ submanifold $N_{\theta}$ and $\partial N_{\theta} = B$, for all $\theta \in \mathbb{S}^1.$ The submanifold $B$ is called the \emph{binding} and $\Sigma = \overline{\N_{\theta}}$ is called the \emph{page}. 
   \end{definition}

  	  \noindent One can see from Definition \ref{def: op book} that $\Sigma$ is a manifold with nonempty boundary and the \emph{monodromy} of the fibration $\pi$, denoted by $\phi$, is a diffeomorphism of $\Sigma$ that is identity on a collar neighborhood of the boundary $\partial \Sigma$. This shows that $M$ can be seen as $\mathcal{MT}(\Sigma, \phi) \cup_{id} \partial \Sigma \times \mathbb{D}^2$, where $\mathcal{MT}(\Sigma, \phi)$ is the mapping torus of $\phi$ and $id$ denotes the identity map of $\partial \Sigma \times \mathbb{S}^1.$ Such a decomposition of $M$, known as \emph{abstract open book}, is equivalent to the above definition of open book and is denoted by $\mathcal{O}b(\Sigma, \phi)$. We will be using this definition of open book throughout the article. Following are some importand facts about open books.

   \begin{enumerate}

   	\item The page $\Sigma$ and the isotopy class of the monodromy $\phi$ uniquely determines the diffeomorphism type of $M$. In particular, we can assume that $\phi$ belongs to the relative diffeotopy group $Diffeo(\Sigma,\partial \Sigma)$.
   	
   	\item If $\phi_1,\phi_2 \in Diffeo(\Sigma,\partial \Sigma)$, then $\mathcal{O}b(V, \phi_1 \circ \phi_2) \cong \mathcal{O}b(V, \phi_2 \circ \phi_1)$.
   	
   	\item In general, two open books $\mathcal{O}b(\Sigma_1,\phi_1)$ and $\mathcal{O}b(\Sigma_2,\phi_2)$ are called \emph{equivalent} if there exists a proper diffeomorphism $h : (\Sigma_1, \partial \Sigma_1) \rightarrow (\Sigma_2, \partial \Sigma_2)$ such that $h \circ \phi_2 = \phi_1 \circ h$.
   	
   \end{enumerate}

   We now define the notion of embedding in the category of open books. Let $M^n$ and $V^N$ be two manifolds admitting open book decompositions. Assume that $N \geq n+1$. 
   
   \begin{definition}[Open book embedding]\label{ob embed def}
   	$M^n$ admits an open book embedding in $V^N$ if there is an open book $\mathcal{O}b(\Sigma_M^{n-1},\phi_M)$ of $M$ and an open book $\mathcal{O}b(\Sigma_V^{N-1},\phi_V)$ of $V$ such that the following conditions hold.
   	
   	\begin{enumerate}
   		\item There exists a proper embedding $f: (\Sigma_M,\partial \Sigma_M) \rightarrow (\Sigma_V,\partial \Sigma_V)$,
   		\item $\phi_V \circ f$ is isotopic (relative to boundary) to $f \circ \phi_M$.
   	\end{enumerate}
   	
   	We say, $M^n$ open book embeds in $V^N$ with respect to the open book $\mathcal{O}b(\Sigma_V,\phi_V)$ \emph{or} $M^n$ open book embeds in $\mathcal{O}b(\Sigma_V,\phi_V)$.  	
   \end{definition}

   \begin{exmp}
   	$\mathbb{S}^n = \mathcal{O}b(\mathbb{D}^{n-1},id)$ admits a canonical open book embedding in $\mathcal{O}b(\mathbb{D}^{N-1},id) = \mathbb{S}^N$ for $N-n \geq 1$. 
   \end{exmp}

   \begin{exmp}
   	Consider the unit disk cotangent bundle $DT^*\mathbb{S}^1$ and let $\tau$ be a Dehn twist along its core circle. It is known that $\mathbb{S}^3 = \mathcal{O}b(DT^*\mathbb{S}^1,\tau^k)$ open book embeds in $\mathcal{O}b(\mathbb{D}^4,id)$ for all $k \in \mathbb{Z}$. For a proof see Corollary $9$ of \cite{PPS}. 
   \end{exmp}

  To prove our results on open book embeddings we shall use the notion of a \emph{flexible embedding}.

  \begin{definition}[Flexible embedding] \label{flex embed}
  	A proper embedding $f: (\Sigma^n,\partial \Sigma^n) \rightarrow (V^{n+k},\partial V^{n+k})$ is called flexible if for every diffeomorphism $\phi$ of $(\Sigma,\partial \Sigma)$ there is an isotopy $\Phi_t$ ($t \in [0,1]$) of $(V^{n+k},\partial V^{n+k})$ such that $\Phi_0 = id$ and $\Phi_1 \circ f = f \circ \phi$. Here $\partial \Sigma$ and $\partial V$ could be empty sets.
  	
  \end{definition}

  \begin{exmp}\label{flex exmp}
  	Let $f$ be an embedding of a closed manifold $M^m$ in $\mathbb{R}^{2m+1}$ and $\phi \in Diffeo(M)$. By Wu \cite{Wu}, $f$ and $f\circ \phi$ are isotopic, say via $\phi_t$. One can then extend this to an ambient isotopy $\Phi_t$ of $\mathbb{R}^{2m+1}$ such that $\Phi_0 = id$ and $\Phi_1 \circ f = f \circ \phi$. Thus, any embedding of $M^m$ in $\mathbb{R}^{2m+1}$ is flexible.
  \end{exmp}

   \subsection{Winkelnkemper open books}
   Winkelnkemper \cite{Wi} showed existence of open book decomposition for simply connected manifolds of dimension greater than $6$. The prototype for Winkelnkemper's construction, as mentioned in \cite{Wi}, is the following example.
   
   \begin{exmp}
   	Let $V$ be any compact manifold with non-empty boundary. Consider the quotient space $W = \frac{V \times [0,1]}{(x,t) \sim (x,\frac{1}{2})}$. Let $N \subset \partial W$ be the image of $\partial V \times [0,1]$ under the quotient map. $N$ then divides $\partial W$ into two parts: $\partial_{up}W$ and $\partial_{lo}W$. If $h$ is a diffeomorphism of the triple $(\partial W, \partial_{up}W, \partial_{lo}W)$ then $W \cup_h W$ has an open book decomposition with binding $N$. See Figure \ref{wink exmp}.
   \end{exmp} 
   
   \begin{figure}[htbp] 
   	
   	\centering
   	\def\svgwidth{8cm}
\begingroup%
  \makeatletter%
  \providecommand\color[2][]{%
    \errmessage{(Inkscape) Color is used for the text in Inkscape, but the package 'color.sty' is not loaded}%
    \renewcommand\color[2][]{}%
  }%
  \providecommand\transparent[1]{%
    \errmessage{(Inkscape) Transparency is used (non-zero) for the text in Inkscape, but the package 'transparent.sty' is not loaded}%
    \renewcommand\transparent[1]{}%
  }%
  \providecommand\rotatebox[2]{#2}%
  \newcommand*\fsize{\dimexpr\f@size pt\relax}%
  \newcommand*\lineheight[1]{\fontsize{\fsize}{#1\fsize}\selectfont}%
  \ifx\svgwidth\undefined%
    \setlength{\unitlength}{313.10594444bp}%
    \ifx\svgscale\undefined%
      \relax%
    \else%
      \setlength{\unitlength}{\unitlength * \real{\svgscale}}%
    \fi%
  \else%
    \setlength{\unitlength}{\svgwidth}%
  \fi%
  \global\let\svgwidth\undefined%
  \global\let\svgscale\undefined%
  \makeatother%
  \begin{picture}(1,0.43717555)%
    \lineheight{1}%
    \setlength\tabcolsep{0pt}%
    \put(0,0){\includegraphics[width=\unitlength,page=1]{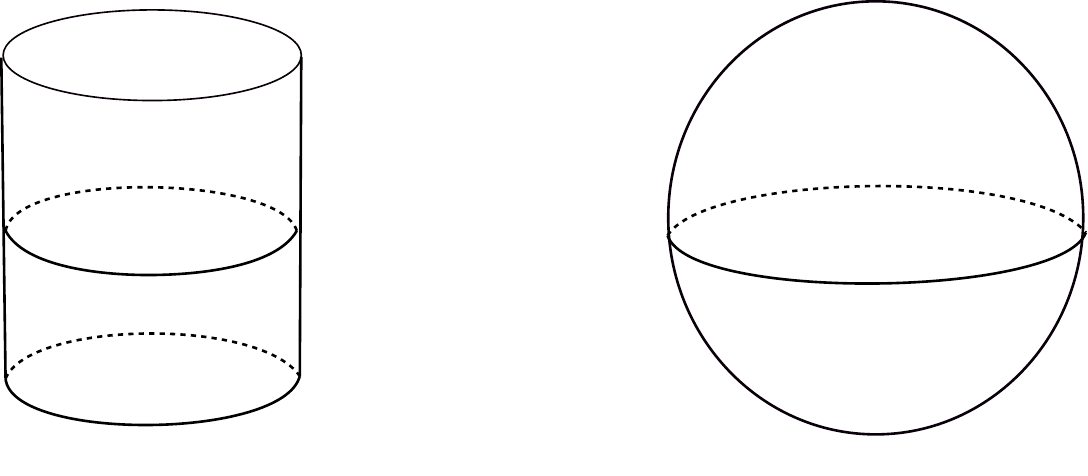}}%
    \put(0.73906112,0.318584){\color[rgb]{0,0,0}\makebox(0,0)[lt]{\lineheight{1.25}\smash{\begin{tabular}[t]{l}$\partial_{up}W$\end{tabular}}}}%
    \put(0.74187869,0.08287819){\color[rgb]{0,0,0}\makebox(0,0)[lt]{\lineheight{1.25}\smash{\begin{tabular}[t]{l}$\partial_{lo}W$\end{tabular}}}}%
    \put(0.78260922,0.17341149){\color[rgb]{0,0,0}\makebox(0,0)[lt]{\lineheight{1.25}\smash{\begin{tabular}[t]{l}$N$\end{tabular}}}}%
    \put(0.06526408,0.00320518){\color[rgb]{0,0,0}\makebox(0,0)[lt]{\lineheight{1.25}\smash{\begin{tabular}[t]{l}$V \times [0,1]$\end{tabular}}}}%
    \put(0.28681247,0.08367255){\color[rgb]{0,0,0}\makebox(0,0)[lt]{\lineheight{1.25}\smash{\begin{tabular}[t]{l}$0$\end{tabular}}}}%
    \put(0.28843355,0.3714254){\color[rgb]{0,0,0}\makebox(0,0)[lt]{\lineheight{1.25}\smash{\begin{tabular}[t]{l}$1$\end{tabular}}}}%
    \put(0.28779658,0.21521786){\color[rgb]{0,0,0}\makebox(0,0)[lt]{\lineheight{1.25}\smash{\begin{tabular}[t]{l}$\frac{1}{2}$\end{tabular}}}}%
    \put(0,0){\includegraphics[width=\unitlength,page=2]{drawing-2.pdf}}%
  \end{picture}%
\endgroup%

   	\caption{Prototype example for Winkelnkemper open book}
   	\label{wink exmp}
   	
   \end{figure}

    We briefly review the idea behind Winkelnkemper's construction. As in \cite{Wi}, we assume that $M^n$ is a simply connected closed $n$-manifold and $n = 2k+1$ for $k \geq 3$. 
   
    \noindent Take a minimal handle decomposition of $M^n$. Let $W_1$ denote the handlebody constructed by attaching handles of dimension up to $k$ and let $W_2$ denote the dual handlebody. Then $M^n = W_1 \cup_h W_2$, where $W_1$ and $W_2$ are attached along their common boundary $E = W_1 \cap W_2 = \partial W_1 = \partial W_2$ via some diffeomorphism $h$ of $E$. The following lemma is a crucial step in Winkelnkemper's argument.
   
   \begin{lemma}[Winkelnkemper \cite{Wi}]\label{assert}
   	There exists a $k$-complex $K \subset E = \partial W_1 = \partial W_2$ such that both inclusions $K \subset W_j$ ($j = 1, 2$) are homotopy equivalences.
   \end{lemma}
   
   Now let $V$ be a regular neighborhood of $K$ in $E$. This $V$ will be a candidate for the page and $\partial V$ will be the binding. Since $K$ has co-dimension $\geq 3$, $V$ and $\partial V$ are simply connected. Take a collar $\partial V \times [0,1]$ of $V$ and regard $W_1$ and $W_2$ as a relative cobordism between $V$ and the closure of the complement of $V \cup \partial V \times [0,1]$ in $E$. The assertion then implies that both $W_1$ and $W_2$ are relative $h$-cobordisms. Then, by the relative $h$-cobordism theorem, $W_1 = V \times [0,1] = W_2$, and by contracting the collar neighborhood $\partial V \times [0,1]$ we get the desired open book with binding $\partial V$ and page $V$. We call this open book the \emph{Winkelnkemper open book}. We observe that if $M^{2k+1}$ is $l$-connected, then by the above construction we get an open book whose pages are also $l$-connected. The reason is that we can then start with a handlebody decomposition of $M^{2k+1}$ such that both $W_1$ and $W_2$ have no handles upto dimension $l$. We note this in the following lemma.
   
   \begin{lemma}\label{le wink}
   	If $M^{2k+1}$ is a closed $l$-connected manifold ($l\geq1$), then the Winkelnkemper open book has $l$-connected pages.
   \end{lemma}
   
   \subsection{Embedding and isotopy of manifolds in Euclidean space}

   Let $M^n$ be a closed $k$-connected manifold. Let $M_0$ denote $M \setminus \{pt.\}$. Haefliger--Hirsch \cite{HH} proved the following generalizations of Whitney's embedding theorem \cite{Wh} and Wu's isotopy theorem \cite{Wu}.
   
   \begin{theorem}[Haefliger--Hirsch \cite{HH}]\label{hh embedding}
   	Assume $0 \leq k < \frac{1}{2}(n-4)$. If $M_0$ can be immersed in $\mathbb{R}^{2n-k-1}$ with a normal vector field, then $M^n$ can be embedded in $\mathbb{R}^{2n-k-1}$.
   \end{theorem}
   
   \begin{theorem}[Haefliger--Hirsch \cite{HH}]\label{hh isotopy}
   	Assume $0 \leq k \leq \frac{1}{2}(n-4)$. If $M^n$ is orientable there is a $1-1$ correspondence between the isotopy classes of embeddings of $M$ in $\mathbb{R}^{2n-k}$ and the regular homotopy classes of immersions of $M_0$ in $\mathbb{R}^{2n-k}$ with a normal vector field.
   \end{theorem}
   
   The proof of Theorem \ref{hh embedding} and Theorem \ref{hh isotopy} is based on the following theorems of Haefliger and Hirsch.
   
   \begin{theorem}[\cite{HH}]\label{hh 3.1}
   	Let $M^n$ be a $k$-connected manifold.
   	\begin{enumerate}
   		\item If $m \geq 2n-k-1$, then $M_0$ can be immersed in $\mathbb{R}^m$, and any immersion is regularly homotopic to an embedding.
   		\item If $m \geq 2n-k$, any two embeddings $f$ and $g$ of $M_0$ in $\mathbb{R}^m$ are regularly homotopic. If $G$ is a regular homotopy connecting $f$ and $g$, there is a regular homotopy $G_t$ of $G$ such that $G_0 = G$, $G_1$ is an isotopy, and for each $t$, $G_t$ connects $f$ to $g$. 
   	\end{enumerate}
   \end{theorem}
   
   \begin{theorem}[\cite{HH}]\label{hh 3.2}
   	
   	Let $X$ be an $m$-manifold and $E$ an open $n$-disk.
   	
   	\begin{enumerate}
   		\item Suppose $2m \geq 3(n+1)$ and $X$ is $(2n-m+1)$-connected. Let $g : E \rightarrow X$ be a proper map whose restriction to the complement of some compact set is an embedding. Then there is a homotopy, fixed outside of a compact set, which deforms $g$ into an embedding.
   		\item Suppose $2m \geq 3(n+1)$ and $X$ is $(2n-m+2)$-connected. Let $g_0 , g_1 : E \rightarrow X$ be proper embeddings which which are connected by a homotopy fixed outside of a compact set. Then $g_0$ and $g_1$ are also connected by an isotopy $g_t$, fixed outside of a compact set.
   	\end{enumerate}
   \end{theorem}

    A straight forward analysis of the proof of Theorem \ref{hh embedding} and Theorem \ref{hh isotopy} reveals that both the theorems hold for proper embeddings of a manifold $(V^n,\partial V^n)$ in an Euclidean disk $(\mathbb{D}^N,\partial \mathbb{D}^N)$. We will later review the proofs of Theorem \ref{hh embedding} and Theorem \ref{hh isotopy} and state their relative versions for manifolds with boundary in Section \ref{k-connect sect}.

   \subsection{Contact manifold and isocontact embedding}  \label{isocontact}

   A \emph{contact manifold} is an odd dimensional smooth manifold $M^{2n+1}$, together with a maximally non-integrable co-dimension $1$ distribution $\xi \subset TM$. A \emph{contact form} $\alpha$ representing $\xi$ is a local $1$-form on $M$ such that $\xi = Ker\{\alpha\}$. The contact condition is equivalent to saying that $\alpha\wedge(d\alpha)^n$ is a volume form. The $2$-form $d\alpha$ induces a conformal symplectic structure on $\xi$. $\xi$ is called \emph{co-orientable} if $TM/\xi$ is a trivial line bundle. Throughout this article we will only consider co-orientable contact structures on closed, orientable manifolds. We will denote a manifold $M$ together with a contact structure $\xi$ by $(M,\xi)$ and use $\xi_{std}$ to denote the standard contact structure on $\mathbb{R}^{2N+1}$ (\emph{or} $\mathbb{S}^{2N+1}$), given by $Ker\{dz + \Sigma^N_{i=1} x_idy_i\}$. Two contact manifolds $(M_1,\xi_1)$ and $(M_2,\xi_2)$ are called \emph{contactomorphic}if there is a diffeomorphism $h$ between them such that $Dh(\xi_1) = \xi_2$. If $\alpha_i$ is a contact form representing $\xi_i$ ($i =1,2$), then $h^*\alpha_2 = u\cdot\alpha_1$ for some positive function on $M_1$. For details on contact manifolds we refer to \cite{Ge}.

   We now define the notion of embedding in the category of contact manifolds.
   
   \begin{definition}[Isocontact embedding] $(M^{2n+1},\xi)$ admits an isocontact embedding in $(V^{2N+1},\eta)$, if there is an embedding $\iota: M \hookrightarrow V$ such that for all $p \in M$, $D\iota(T_pM)$ is transverse to $\eta_{f(p)}$ and $D\iota(T_pM)\cap{\eta}_{f(p)} = D\iota({\xi}_p)$.
   \end{definition}	
   
   \noindent If $\beta$ is a contact form representing $\eta$, then $D\iota(\xi)$ is a conformal symplectic sub-bundle of $(\eta|_{\iota(M)},d\beta)$. Gromov \cite{Gr} proved the following contact analog of the Whitney embedding theorem.
   
   \begin{theorem}[Gromov \cite{Gr}]\label{h-principle contact embed}
   	Every contact manifold $(M^{2n+1},\xi)$ has an isocontact embedding in $(\mathbb{R}^{4n+3},\xi_{std})$.
   \end{theorem}

   	\noindent When the co-dimension of a topological embedding of $(M,\xi)$ in an Euclidean space is less than $dim(M)-1$, the Chern classes of $\xi$ also give obstructions to isocontact embedding. For example, a necessary condition to isocontact embed a $3$-manifold $(M_0^3,\xi_0)$ in $(\mathbb{R}^5,\xi_{std})$ is $c_1(\xi_0) = 0$. By the recent works of Casals--Pancholi--Presas \cite{CPP}, it is now also sufficient.

   	\subsection{Almost contact structure}
   	We recall the notion of almost contact structure.

   	\begin{definition} An almost contact structure on an odd dimensional manifold $N^{2n+1}$ is an almost complex structure on its stable tangent bundle $TN\oplus\varepsilon^1_N$.
   	\end{definition}
   	
   	Thus, an almost contact structure on $N$ is an almost complex structure on $N \times \mathbb{R}$. So every almost contact structure on $N$ is given by a section of the associated $\Gamma_{n+1}$--bundle of $T(N\times\mathbb{R})$. The existence of an almost contact structure on $N$ is a necessary condition for the existence of a contact structure. For open manifolds, Gromov (\cite{Gr}) proved the following h-principle showing that this condition is also sufficient. 
   	
   	\begin{theorem} (Gromov, \cite{Gr}) \label{existence of contact structure gromov}
   		Let $K$ be a sub-complex of an open manifold $V$. Let $\bar{\xi}$ be an almost contact structure on $V$ which restricts to a contact structure in a neighborhood $Op(K)$ of $K$. Then one can homotope $\bar{\xi}$, relative to $Op(K)$, to a contact structure $\xi$ on $V$.    
   	\end{theorem}

   \subsection{Contact open book and embedding}
   
   Let $(V,\partial V,d\alpha)$ be an exact symplectic manifold which has a collar neighborhood symplectomorphic to $((-1,0] \times \partial V,d(e^t \cdot \alpha))$ for $t \in (-1,0]$. The Liouville vector field $Y$ for $d \alpha$ is defined by $\iota_Yd\alpha = \alpha$. $Y$ looks like $\frac{\partial}{\partial t}$ near $\partial V$ and therefore transverse to $\partial V$, pointing outwards. The $1$-form $e^t \cdot \alpha$ induces a contact structure on $\partial V$. Let $\phi$ be a symplectomorphism of $(V,d\alpha)$ that is identity in a collar of $\partial V$. The following lemma, due to Giroux, says that we can assume $\phi^*\alpha - \alpha$ to be exact.

   \begin{lemma}[Giroux \cite{Ko}]
   	The symplectomorphism $\phi$ of $(V,d\alpha)$ is isotopic, via symplectomorphisms which are identity near $\partial V$, to a symplectomorphism $\phi_1$ such that $\phi_1^*\alpha - \alpha$ is exact.
   \end{lemma}

   Let $\phi^*\alpha - \alpha = dh$. Here $h: V \rightarrow \mathbb{R}$ is a function well defined up to addition by constants. Note that $dt + \alpha$ is a contact form on $\mathbb{R} \times V$, where the $t$ co-ordinate is taken along $\mathbb{R}$. Take the mapping torus $\mathcal{MT}(V,\phi)$ defined by the following map.
   
   \begin{alignat*}{2}
   \Delta: (\mathbb{R} \times V, dt + \alpha) &\longrightarrow& (\mathbb{R} \times V, dt + \alpha) \\
   (t,x) &\longmapsto& (t-h,\phi(x)) 
   \end{alignat*}
   
   \noindent The contact form $dt + \alpha$ then descends to a contact form $\lambda$ on $\mathcal{MT}(V,\phi)$. Since $\phi$ is identity near $\partial V$, a contact neighborhood of the boundary of $\mathcal{MT}(V,\phi)$ looks like $((-\frac{1}{2},0) \times \partial V \times \mathbb{S}^1, e^r\cdot\alpha|_{\partial V} + dt)$. Let $A(r, R)$ denote the annulur region given by $\{z \in \mathbb{C} \ | \ r < |z| < R\}$. Define $\Phi$ as follows.

   \begin{alignat*}{2}
   \Phi: \partial V \times A(\frac{1}{2},1) &\longrightarrow& (-\frac{1}{2},0) \times \partial V \times S^1 \\
   (v, r e^{it}) &\longmapsto& (\frac{1}{2} - r, v, t) 
   \end{alignat*}
   
   \noindent We then glue $\mathcal{MT}(V,\phi)$ and $\partial V \times \mathbb{D}^2$ using the map $\Phi$. Note that $\Phi$ pulls back $\lambda$ to $(e^{\frac{1}{2}- r} \cdot \alpha|_{\partial V} + dt)$ on $\partial V \times A(\frac{1}{2}, 1)$. One can then extend this pulled back form to the interior of $\partial V \times \mathbb{D}^2$ a by defining $\beta = h_1(r)\cdot \alpha|_{\partial V} + h_2(r)\cdot dt$, where $h_1$ and $h_2$ are as in Figure \ref{h1h2}. This defines a contact structure on $W^{2n+1} = \mathcal{MT}(V,\phi) \cup_{id} \partial V \times \mathbb{D}^2$ which coincides with $\lambda$ on $\mathcal{MT}(V,\phi)$ and with $\alpha + r^2 dt$ on $\partial V \times \mathbb{D}^2$. We denote the resulting contact manifold $(W^{2n+1},\beta)$ by $\mathcal{O}b(V,d\alpha;\phi)$.
   Note that the contactomorphism type of $\mathcal{O}b(V,d\alpha;\phi)$ is determined by $(V,d\alpha)$ and the symplectic isotopy class of $\phi$.
   
   \begin{figure}[htbp]
  	
  	\begin{tikzpicture}
  	\draw[thick,->] (0,0) -- (5,0) node[anchor=north east] {$r$};
  	\draw[thick,->] (0,0) -- (0,4) node[anchor=north east] {$h_1(r)$}; 
  	\draw (3.5 cm,1pt) -- (3.5 cm,-1pt) node[anchor=north] {$\frac{1}{2}$};
  	\draw[very thick] (0,3) -- (0.5,3);
  	\draw[very thick] (0.5,3) parabola (2,2.5);
  	\draw[very thick] (4.2,1.5) parabola (2,2.5);
  	\draw[black,thick,dashed] (3.5,0) -- (3.5,4);
  	\draw (0,3) node[anchor=east]{$h_1(0)$};
  	\draw[black,thick,dashed] (0,3) -- (4,3);

  	\draw[thick,->] (8,0) -- (13,0) node[anchor=north east] {$r$};
  	\draw[thick,->] (8,0) -- (8,4) node[anchor=north east] {$h_2(r)$}; 
  	\draw (11.5 cm,1pt) -- (11.5 cm,-1pt) node[anchor=north] {$\frac{1}{2}$};
  	\draw[very thick] (8,0) -- (8.5,0);
  	\draw[very thick] (8.5,0) parabola (9.5,2);
  	\draw[very thick] (10.5,3) parabola (9.5,2);
  	\draw[very thick] (10.5,3) -- (12,3);
  	\draw[black,thick,dashed] (11.5,0) -- (11.5,3.5);
  	\draw[black,thick,dashed] (8,3) -- (12,3);
  	
  	\end{tikzpicture}	
  	\caption[]{Functions for the contact form near binding}
  	\label{h1h2}
  \end{figure}
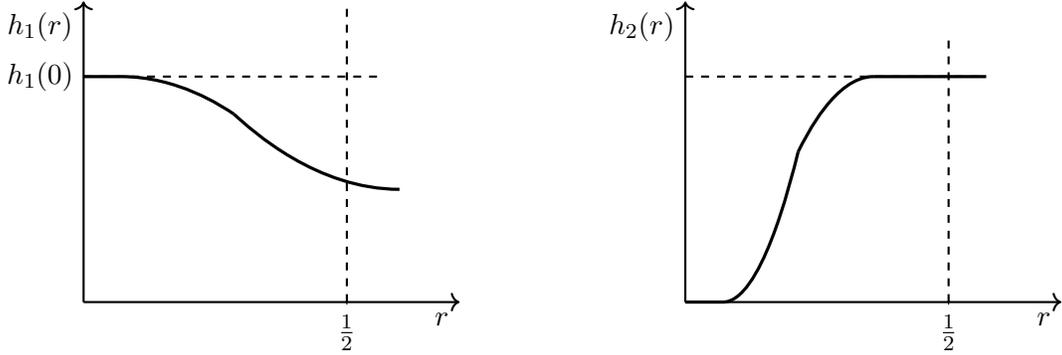

   \begin{definition}[Contact open book]
   	$\mathcal{O}b(V, d\alpha,\phi)$ is called a \emph{contact open book decomposition} with \emph{page} $(V,d\alpha)$ and \emph{binding} $(\partial V,\alpha)$. Given a manifold $M$ with a contact form $\beta$ on it, if one can find an open book $\mathcal{O}b(V_M,d\alpha_M;\phi_M)$ that is contactomorphic to $(M,\beta)$, then one says that $\mathcal{O}b(V_M,\alpha_M;\phi_M)$ is an open book decomposition of $M$ \emph{supporting} the contact form $\beta$. 
   	
   \end{definition}
   
   \noindent Given $(M,\xi)$, if there exists a contact form $\beta$ representing $\xi$ such that $(M,\beta)$ has a supporting open book, then we say that $(M,\xi)$ has a supporting open book. In particular, we will write $(M,\xi) = \mathcal{O}b(V_M,d\alpha_M;\phi_M)$ to say that $(M,\xi)$ is supported by an open book with page $(V_M,d\alpha_M)$ and monodromy $\phi_M$. Giroux \cite{Gi} has proved that every contact manifold admits a supporting open book.

   \begin{definition}[Contact open book embedding]\label{contact open book embed def}
   	$(M^{2n+1}_1,\xi_1)$ \emph{contact open book embeds} in $(M^{2N+1}_2,\xi_2)$ if there exist contact open books $\mathcal{O}b(\Sigma_i,d\alpha_i,\phi_i)$  supporting $(M_i^{2n+1},\xi_i)$, for $i = 1,2$, such that the following conditions hold.
   	
   	\begin{enumerate}
   		\item There exists a proper isosymplectic embedding $g : (\Sigma_1, \partial \Sigma_1, d\alpha_1) \rightarrow (\Sigma_2,\partial \Sigma_2, d\alpha_2)$, i.e., $g^*d\alpha_2 = d\alpha_1.$
   		\item$g \circ \phi_1 = \phi_2 \circ g$ upto contact isotopy.
   	\end{enumerate}  
   \end{definition}
   
   \noindent Definition \ref{contact open book embed def} implies that $\mathcal{MT}(\Sigma_1,\phi_1)$ isocontact embeds in $\mathcal{MT}(\Sigma_2,\phi_2)$. Since $g|_{\partial \Sigma_1}$ pulls back the contact form $\alpha_2$ to $u \cdot \alpha_1$, for some positive function $u$ on $\partial \Sigma_1$, we can extend this embedding to an isocontact embedding $\tilde{g}$ of $\mathcal{O}b(\Sigma_1,d\alpha_1,\phi_1)$ in $\mathcal{O}b(\Sigma_2,d\alpha_2,\phi_2)$ such that the restriction of $\mathcal{O}b(\Sigma_2,d\alpha_2,\phi_2)$ on the image $Im(\tilde{g})$ gives the contact open book $\mathcal{O}b(\Sigma_1,d\alpha_1,\phi_1)$.

   \subsection{Dehn-Seidel twist} \label{dehn twist} 
   
   Consider the symplectic structure on the cotangent bundle $(T^*\mathbb{S}^n,d\lambda_{can})$. Here, $\lambda_{can}$ is the canonical 1-form on $T^*\mathbb{S}^n$. In local coordinates $(x_1,x_2,...,x_{n+1},y_1,y_2,...,y_{n+1})$, $\lambda_{can}$ is given by the form $\sum y_idx_i$. Here, we regard $T^*\mathbb{S}^n$ as a submanifold of $\mathbb{R}^{2n+2} \cong \mathbb{R}^{n+1} \times \mathbb{R}^{n+1}$. A point $(\vec{x},\vec{y}) \in \mathbb{R}^{n+1} \times \mathbb{R}^{n+1}$, represents a point in $T^*\mathbb{S}^n$ if and only if it satisfies the relations: $|\vec{x}| = 1$ and $\vec{x} \cdot \vec{y} = 0$. Here, $\vec{y} \equiv (y_1,..,y_{n+1})$ and $\vec{x} \equiv (x_1,..,x_{n+1})$. Let $\sigma_t : T^*\mathbb{S}^n \rightarrow T^*\mathbb{S}^n$ be defined as follows.
   
   \begin{equation*}
   \sigma_t(\vec{x},\vec{y}) = \begin{pmatrix}
   \cos t&|\vec{y}|^{-1}\sin t\   \\  -|\vec{y}|\sin t&\cos t 
   \end{pmatrix}
   \begin{pmatrix}
   \vec{x} \ \\ \vec{y}
   \end{pmatrix}
   \end{equation*}

   \noindent For $k \in \mathbb{Z}_{>0}$, let $g_k : [0,\infty) \rightarrow \mathbb{R}$ be a smooth function that satisfies the following properties.
   
   \begin{enumerate}
   	\item $g_k(0) = k\pi$ and ${g'_k}(0) < 0$.
   	\item Fix $p_0 > 0$. The function $g_k(|\vec{y}|)$ decreases to $0$ at $p_0$ and then remains $0$ for all $\vec{y}$ with $|\vec{y}|>p_0$. See Figure \ref{dehntwistpic} below. 
   \end{enumerate}

   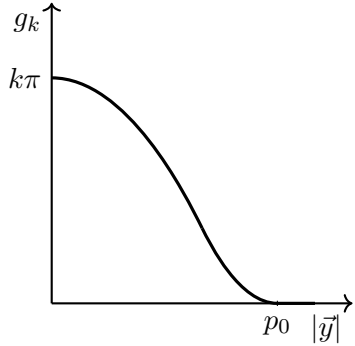
\begin{figure}[htbp]
   	
   	\begin{tikzpicture}
   	
   	\draw[thick,->] (0,0) -- (4,0) node[anchor=north east] {$|\vec{y}|$} ;
   	\draw[thick,->] (0,0) -- (0,4)node[anchor=north east] {$g_k$}; 
   	\draw[very thick] (0,3) parabola  (2,1);
   	\draw[very thick](3,0) parabola (2,1);
   	\draw[very thick] (3,0) -- (3.5,0);
   	\draw (3 cm,1pt) -- (3 cm,-1pt) node[anchor=north] {$p_0$};
   	
   	\draw (0,3) node[anchor=east] {$k\pi$};
   	
   	\end{tikzpicture}
   	\caption{Cut-off function for defining $k$-fold Dehn-Seidel twist.}	
   	\label{dehntwistpic}
   	
   \end{figure}

   Now we can define the \textit{positive $k$-fold Dehn-Seidel twist} as follows.
   
   $$\tau_k(\vec{x},\vec{y}) = \begin{cases}
   \sigma_{g_k(|\vec{y}|)}(\vec{x},\vec{y})  \  \   for \ \ \vec{y} \neq \vec{0}\\
   -Id  \  \ for \ \ \vec{y} = \vec{0}
   \end{cases}$$

   \noindent The Dehn-Seidel twist is a proper symplectomorphism of $T^*\mathbb{S}^n$. From Figure \ref{dehntwistpic}, we see that $\tau_k$ has compact support. Therefore, choosing $p_0$ properly, $\tau_k$ can be defined on the unit disk bundle $(DT^*\mathbb{S}^n,d\lambda_{can})$, such that it is identity near boundary. In fact, we can choose the support as small as we wish without affecting the symplectic isotopy class of the resulting $\tau_k$.

   Similarly, for $k<0$, we can define the \textit{negative $k$-fold Dehn-Seidel twist}. For $k = 0$, $\tau_0$ is defined to be the identity map of $DT^*\mathbb{S}^n$. Sometimes we may say just \emph{Dehn twist} instead of \emph{Dehn-Seidel twist}.
   
   \begin{exmp}
   	An important open book of $(\mathbb{S}^{2n+1},\xi_{std})$ is given with page $(DT^*\mathbb{S}^n,d\lambda^n_{can})$ and monodromy a positive Dehn-Seidel twist $\tau_n$. We call this the \emph{standard open book} of $\mathbb{S}^{2n+1}$.
   \end{exmp}

   \subsection{Lefschetz fibrations on Weinstein domains} \label{weinstein} We now recall the notion of an \emph{abstract Weinstein Lefschetz fibration} as defined by Giroux--Pardon \cite{GP}. For details on Weinstein manifolds/domains we refer to \cite{CE}.

   Let $\lambda$ be a $1$-form on a manifold $(V,\partial V)$ such that $\omega = d\lambda$ is symplectic. $\lambda$ is called a \emph{Liouville form} and the vector field $X$, defined by $\iota_X \omega = \lambda$, is called the \emph{Liouville field} of $\lambda$. The pair $(V,\lambda)$ is called an \emph{exact symplectic manifold}. One can also describe $(V,\lambda)$ by the triple $(V, \omega, X)$. The union of the components of $\partial V$ where $X$ points outward (inward) is denoted by $\partial_+V$($\partial_-V$). $\partial_+V$ ($\partial_-V$) is called the \emph{convex boundary} (\emph{concave boundary}) of the \emph{Liouville cobordism} given by $(V, \omega, X)$ (see Figure \ref{livcob}).

   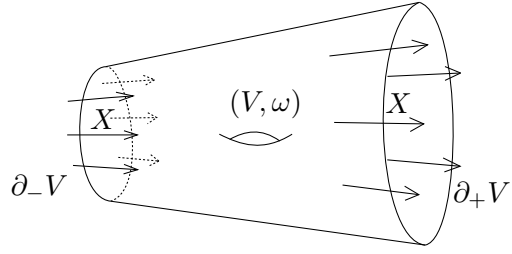
\begin{figure}[htbp] 
   	\centering
   	\def\svgwidth{7cm}
\begingroup%
  \makeatletter%
  \providecommand\color[2][]{%
    \errmessage{(Inkscape) Color is used for the text in Inkscape, but the package 'color.sty' is not loaded}%
    \renewcommand\color[2][]{}%
  }%
  \providecommand\transparent[1]{%
    \errmessage{(Inkscape) Transparency is used (non-zero) for the text in Inkscape, but the package 'transparent.sty' is not loaded}%
    \renewcommand\transparent[1]{}%
  }%
  \providecommand\rotatebox[2]{#2}%
  \newcommand*\fsize{\dimexpr\f@size pt\relax}%
  \newcommand*\lineheight[1]{\fontsize{\fsize}{#1\fsize}\selectfont}%
  \ifx\svgwidth\undefined%
    \setlength{\unitlength}{381.98390621bp}%
    \ifx\svgscale\undefined%
      \relax%
    \else%
      \setlength{\unitlength}{\unitlength * \real{\svgscale}}%
    \fi%
  \else%
    \setlength{\unitlength}{\svgwidth}%
  \fi%
  \global\let\svgwidth\undefined%
  \global\let\svgscale\undefined%
  \makeatother%
  \begin{picture}(1,0.46325245)%
    \lineheight{1}%
    \setlength\tabcolsep{0pt}%
    \put(0,0){\includegraphics[width=\unitlength,page=1]{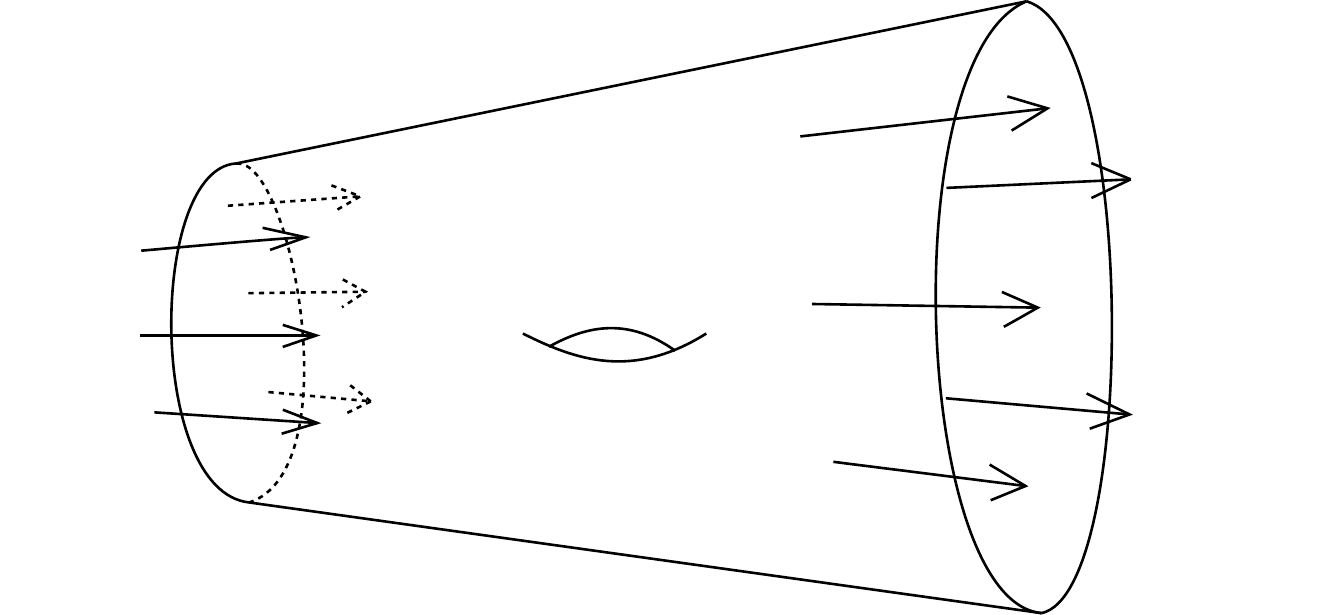}}%
    \put(0.4129622,0.25588482){\color[rgb]{0,0,0}\makebox(0,0)[lt]{\lineheight{1.25}\smash{\begin{tabular}[t]{l}$(V,\omega)$\end{tabular}}}}%
    \put(-0.00198138,0.09625504){\color[rgb]{0,0,0}\makebox(0,0)[lt]{\lineheight{1.25}\smash{\begin{tabular}[t]{l}$\partial_-V$\end{tabular}}}}%
    \put(0.83840145,0.08334475){\color[rgb]{0,0,0}\makebox(0,0)[lt]{\lineheight{1.25}\smash{\begin{tabular}[t]{l}$\partial_+V$\end{tabular}}}}%
    \put(0.14898531,0.2222814){\color[rgb]{0,0,0}\makebox(0,0)[lt]{\lineheight{1.25}\smash{\begin{tabular}[t]{l}$X$\end{tabular}}}}%
    \put(0.71007238,0.25251551){\color[rgb]{0,0,0}\makebox(0,0)[lt]{\lineheight{1.25}\smash{\begin{tabular}[t]{l}$X$\end{tabular}}}}%
  \end{picture}%
\endgroup%

   	\caption{Liouville cobordism with concave boundary $\partial_-V$ and convex boundary $\partial_+V$.}
   	\label{livcob}
   \end{figure}

   \begin{definition}[Weinstein manifolds and Weinstein domains]

   	A \emph{Weinstein manifold} $(W,\omega,X,\phi)$ is a symplectic manifold $(W,\omega)$ with a complete Liouville vector field $X$ which is gradient-like for an exhausting Morse function $\phi : W \rightarrow \mathbb{R}$. A \emph{Weinstein cobordism} $(W_1,\omega,X,\phi)$ is a Liouville cobordism $(W,\omega,X)$ whose Liouville field $X$ is gradient-like for a Morse function $\phi$ on $W_1$ which is constant on the boundary. A Weinstein cobordism with $\partial_-W_1 = \emptyset$ is called a \emph{Weinstein domain}. 
   	
   \end{definition}

   \begin{definition}[Abstract Weinstein Lefschetz fibration]
	An abstract Weinstein Lefschetz fibration is a tuple $W = (W_0 ; L_1 ,\dots, L_m)$
	consisting of a Weinstein domain $W_0^{2n}$ (the “central fiber”) along with a finite sequence of
	exact parameterized Lagrangian spheres $L_1,\dots, L_m \subset W_0$ (the “vanishing cycles”).
	
	Given an abstract Weinstein Lefschetz fibration $W = (W_0; L_1,\dots, L_m)$, one can construct a Weinstein domain $W^{2n+2}$ (its “total space”) by attaching critical Weinstein handles to the
	stabilization $W_0 \times \mathbb{D}^2$ along Legendrians $\Lambda_j \subset W_0 \times S^1 \subset \partial(W_0 \times \mathbb{D}^2)$ near $\frac{2\pi j}{m} \in \mathbb{S}^1$, obtained by lifting the exact Lagrangians $L_j$.
\end{definition}

So an abstract Weinstein Lefschetz fibration $W^{2n+1}$ is a Lefschetz fibration over $\mathbb{D}^2$ with fiber $W_0^{2n}$ and monodromy positive Dehn twists along the vanishing cycles: $L_1,\cdots L_m.$

\begin{definition}[Achiral Lefschetz fibration] \label{achiral LF}
	An achiral Lefschetz fibration $V^{2n+2}$ is just like a Lefschetz fibrtion over $\mathbb{D}^2$ with Weinstein page $V^{2n}_0$. But the monodromy $\phi$ can be represented by composing both positive and negative Dehn-Seidel twists along some Lagrangian spheres in $V^{2n}_0$.
	
	We denote such an achiral Lefschetz fibration by $LF(V^{2n}_0,\phi)$.

\end{definition} 

 \begin{definition}[Deformation equivalence]
	
	Two Weinstein domains $(W_1,J_1,\phi_1)$ and $(W_2,J_2,\phi_2)$ are called deformation equivalent (or can be \emph{deformed} from one to another) if there is a diffeomorphism $h : W_1 \rightarrow W_2$ such that the Weinstein structures $(J_1,\phi_1)$ and $(h^*J_2, h^*\phi_2)$ on $W_1$ are homotopic.
	
\end{definition}

Giroux and Pardon have proved that every Weinstein domain can be deformed to an abstract Weinstein Lefschetz fibration.

\begin{theorem}[Giroux--Pardon,\cite{GP}] \label{GP}
	
	Let $W$ be a Weinstein domain. There exists an abstract Weinstein Lefschetz fibration $\tilde{W} = (W_0; L_1,\dots, L_m)$ whose total space (which we again denote by $\tilde{W}$) is deformation equivalent to $W.$
	
\end{theorem}

Thus, Theorem \ref{GP} shows that any manifold bounding a Weinstein domain has an induced open book decomposition coming from an abstract Weinstein Lefschetz fibration on the domain.

\section{Open book embedding of $k$-connected manifolds in $\mathbb{S}^{2n-k}$}\label{k-connect sect}

In this section, we prove Theorem \ref{open book thm 1}. We want to use a relative version of Theorem \ref{hh embedding} and Theorem \ref{hh isotopy}, which can be stated as follows. 

\begin{lemma}\label{relative HH isotopy}
	Let $(V^n, \partial V^n)$ be a $k$-connected manifold. Let $V_0$ denote the manifold obtained by removing a small $n$-disk from the interior of $V$.
	
	\begin{enumerate}
		
		\item If $V_0$ can be immersed in $\mathbb{R}^{2n-k-1}$ with a normal vector field, then $V^n$ admits a proper embedding in $(\mathbb{D}^{2n-k}, \partial \mathbb{D}^{2n-k})$.
		
		\item There is a $1-1$ correspondence between the isotopy classes of proper embeddings of $V$ in $\mathbb{D}^{2n-k}$ and the regular homotopy classes of proper immersions of $V_0$ in $(\mathbb{D}^{2n-k}, \partial \mathbb{D}^{2n-k})$ with a normal vector field.
		
		\item Any two proper embeddings $f$ and $g$ of $V^n$ in $\mathbb{D}^{2n-k+1}$ are isotopic. 
		
	\end{enumerate}
	
\end{lemma}

The proof of Lemma \ref{relative HH isotopy} uses ideas similar to the proof of Theorem \ref{hh embedding} and Theorem \ref{hh isotopy} work. Let us review the main steps behind these proofs from \cite{HH}.

\subsection{Proofs of Theorem \ref{hh embedding} and Theorem \ref{hh isotopy}}

\begin{proof}[Proof of Theorem \ref{hh embedding}:] 
	Recall that $M^n_0 = M^n \setminus \{x_0\}$, where $x_0$ is a point in $M^n$. Let $f_0 : M^n_0 \rightarrow \mathbb{R}^{2n-k-1}$ be an immersion with a normal vector field $v$. By Theorem \ref{hh 3.1}, $f$ is regularly homotopic to an embedding. So we may assume $f_0$ to be an embedding. Let $D_i$ be an embedded closed disk of radius $i$ around a point $x_0 \in M^n$ for $i = 1,2$. Let $M_i = M \setminus int(D_i)$. The claim is that $f(\partial M_1)$ is an $(n-1)$-sphere homotopic to zero in $X = \mathbb{R}^{2n-k-1} \setminus f(M_2)$.

	Consider an $\epsilon$-tubular neighborhood
	of $f(M_1)$ in $\mathbb{R}^{2n-k-1}$. Let $\lambda : M \rightarrow [0,\epsilon]$ be a smooth function that has value $\epsilon$ on $M_2$ and is $0$ on $D_1$. Then $f(\partial M_1)$ bounds the image of $M_1$ under the map $f_v : M \rightarrow \mathbb{R}^{2n-k-1}$ given by $x \mapsto f(x) + \lambda(x)v(x).$ Thus, $f(\partial M_1)$ is null-homologous in $X$.

	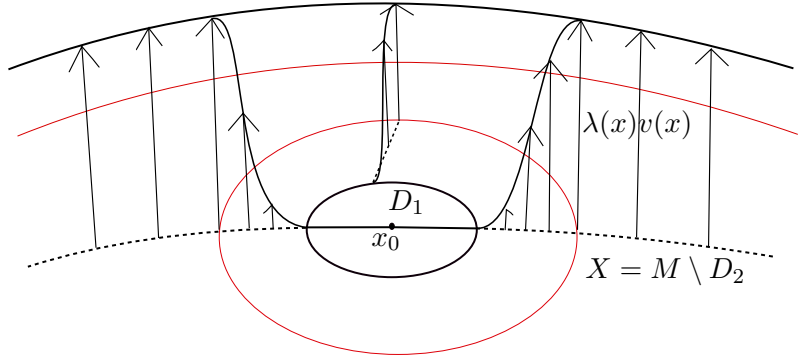
\begin{figure}[htbp]
		
		\def\svgwidth{10.5cm}
\begingroup%
  \makeatletter%
  \providecommand\color[2][]{%
    \errmessage{(Inkscape) Color is used for the text in Inkscape, but the package 'color.sty' is not loaded}%
    \renewcommand\color[2][]{}%
  }%
  \providecommand\transparent[1]{%
    \errmessage{(Inkscape) Transparency is used (non-zero) for the text in Inkscape, but the package 'transparent.sty' is not loaded}%
    \renewcommand\transparent[1]{}%
  }%
  \providecommand\rotatebox[2]{#2}%
  \newcommand*\fsize{\dimexpr\f@size pt\relax}%
  \newcommand*\lineheight[1]{\fontsize{\fsize}{#1\fsize}\selectfont}%
  \ifx\svgwidth\undefined%
    \setlength{\unitlength}{523.38968493bp}%
    \ifx\svgscale\undefined%
      \relax%
    \else%
      \setlength{\unitlength}{\unitlength * \real{\svgscale}}%
    \fi%
  \else%
    \setlength{\unitlength}{\svgwidth}%
  \fi%
  \global\let\svgwidth\undefined%
  \global\let\svgscale\undefined%
  \makeatother%
  \begin{picture}(1,0.44651273)%
    \lineheight{1}%
    \setlength\tabcolsep{0pt}%
    \put(0,0){\includegraphics[width=\unitlength,page=1]{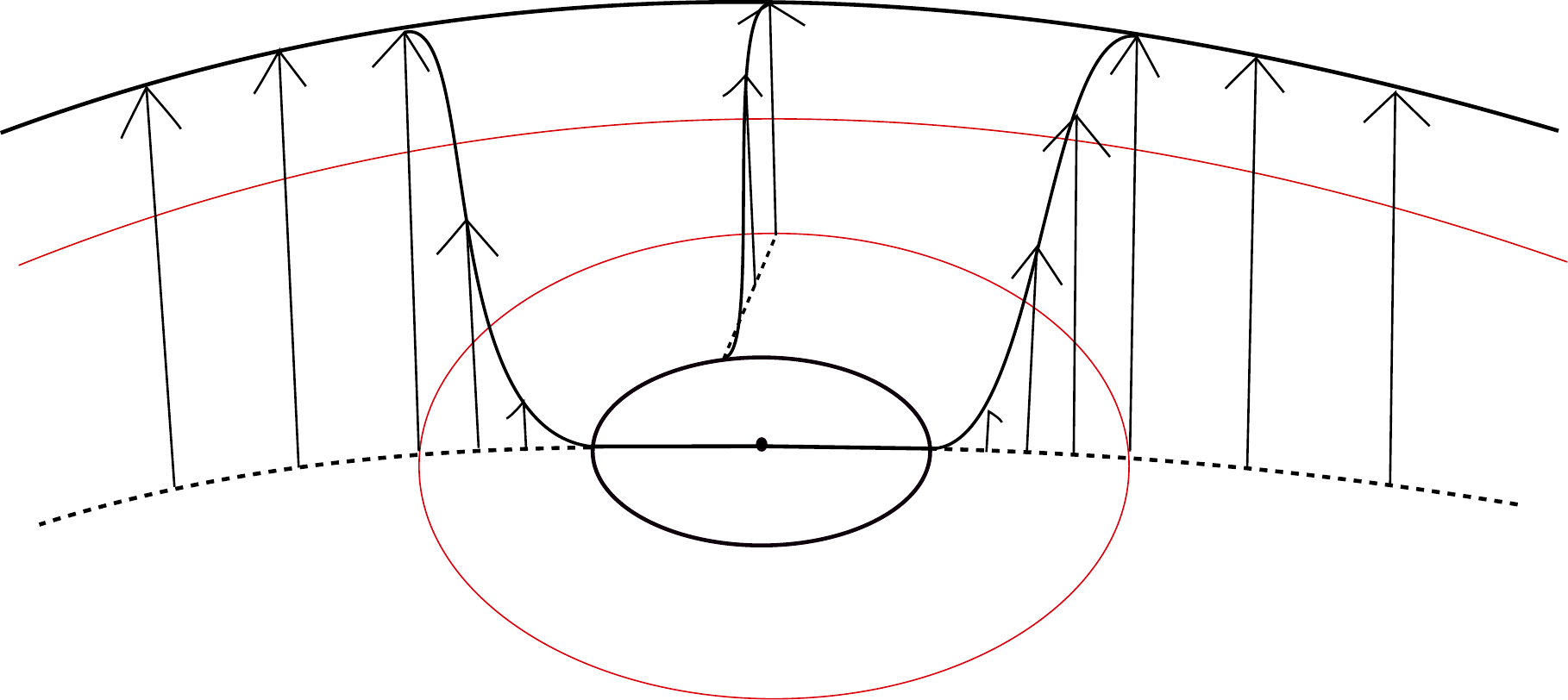}}%
    \put(0.46019438,0.13817485){\color[rgb]{0,0,0}\makebox(0,0)[lt]{\lineheight{1.25}\smash{\begin{tabular}[t]{l}$x_0$\end{tabular}}}}%
    \put(0.48110716,0.18345491){\color[rgb]{0,0,0}\makebox(0,0)[lt]{\lineheight{1.25}\smash{\begin{tabular}[t]{l}$D_1$\end{tabular}}}}%
    \put(0.73114679,0.09688106){\color[rgb]{0,0,0}\makebox(0,0)[lt]{\lineheight{1.25}\smash{\begin{tabular}[t]{l}$X = M \setminus D_2$\end{tabular}}}}%
    \put(0.7269576,0.28281984){\color[rgb]{0,0,0}\makebox(0,0)[lt]{\lineheight{1.25}\smash{\begin{tabular}[t]{l}$\lambda(x)v(x)$\end{tabular}}}}%
  \end{picture}%
\endgroup%

		\caption{Isotoping the embedding of $M_1$ by the vector field $\lambda(x) \cdot v(x)$. The region enclosed between the two red curves is $X = M \setminus D_2$.}
		\label{hpic}
		
	\end{figure}

	\noindent By the Poincare duality and the Alexander duality, $H_i(X) \cong H_{k+i+2-n}(M)$ (with $\mathbb{Z}$-coefficients). Since $M$ is $k$-connected, $H_i(X) = 0$ for $i \leq n-2$. Thus, by the Hurewicz isomorphism between $\pi_{n-1}(X)$ and $H_{n-1}(X)$ we get that that $f(\partial M_1)$ is null-homotopic in $X$.
	
	One can therefore, extend the map $f|_{M_1}$ to a map $g : M \rightarrow \mathbb{R}^{2n-k-1}$ such that $g(M_2) \cap g(int(D_2)) = \emptyset$. An application of Theorem \ref{hh 3.2} then leads to an embedding of $D_2$ in $X$ relative to boundary. Together with the embedding $f|_{M_2}$, this gives an embedding of $M$ in $\mathbb{R}^{2n-k-1}$.

\end{proof}

Next we discuss the proof of Theorem \ref{hh isotopy}. We shall continue with the notations used above.

\begin{proof}[Proof of Theorem \ref{hh isotopy}:]
	
	First one shows that Given an embedding $f : M \rightarrow \mathbb{R}^{2n-k}$, one can associate a normal vector field $v$ on $f(M_0)$ such that $f_v(M)$ is null-homologous in $X = \mathbb{R}^{2n-k-1} \setminus f(M_2)$. Here, $f_v : M \rightarrow \mathbb{R}^{2n-k-1}$ is an embedding which is equal to $f(x) + \lambda(x)v(x)$ for $x \in M_1$ and equal to $f(x)$ for $x \in D_1$. 
	
	An argument similar to the previous proof shows that $X$ is $(n-1)$-connected and $\pi_n(x) = H_n(X) = H^{n-k-1}(M_2)$. If $v_1$ and $v_2$ are two normal (to $f(M_2)$) vector fields, then the difference class $d(v_1,v_2) \in H^{n-k-1}(M_2)$ corresponds to the homology class $[f_{v_1}(M)] - [f_{v_2}(M)] \in H_n(X)$. Also, the homotopy classes of normal vector fields on $f(M_0)$ are in $1-1$ correspondence with $H^{n-k-1}(M_0) = H_n(X)$. Therefore, up to homotopy, there is only one vector field such that $f_v(M)$ is null-homologous in $X$.
	
	The next step is to show that this correspondence is $1-1$. We skip the proof of surjectivity and focus only on the proof of injectivity.
	
	Let $f_1, f_2 : M \rightarrow \mathbb{R}^{2n-k-1}$ be two embeddings with corresponding vector fields $v_1$ and $v_2$, respectively. If $(f_1|_{M_0}, v_1)$ and $(f_2|_{M_0}, v_2)$ are regularly homotopic, then by Theorem \ref{hh 3.1}, they can be assumed to be isotopic. Let $h_t$ be an isotopy joining $f_1|_{M_0}$ and $f_2|_{M_0}$, and let $u_t$, be a normal vector field of $h_t(M_0)$ joining $v_1$ and $v_2$. An isotopy of $h_0(M_0)$ can be extended to an isotopy of $\mathbb{R}^{2n-k-1}$. So one can assume that $f_1$ and $f_2$ agree on $M_1$ and that $v_1 = v_2$. Now observe that $f_{v_1}(M)$ and $f_{v_2}(M)$ are homologous in $X$. Therefore, by similar arguments as before $f_{v_1}(M)$ and $f_{v_2}(M)$ are homotopic in $X$. Thus, $f_1|_{D_1}$ and $f_2|_{D_1}$ are homotopic in $X$. Therefore, by Theorem \ref{hh 3.2} they are isotopic in $X$ by an isotopy fixed on a neighborhood of $\partial D_1$. Hence, $f_1$ and $f_2$ are isotopic.
	
\end{proof}

We now prove Lemma \ref{relative HH isotopy}.

\begin{proof}[Proof of Lemma \ref{relative HH isotopy}]
	
	\begin{enumerate}
		\item First we embed $V$ in the boundary $(2n-k-1)$-sphere of $\mathbb{D}^{2n-k}$ by Theorem \ref{hh embedding}. Let $f_0$ be the embedding. Now define a smooth function $\psi : V \rightarrow [0,\epsilon]$ such that $\psi$ is zero in a neighborhood of $\partial V$. Let $R$ be the radial unit vector field on $\mathbb{D}^{2n-k}$ pointing inward from boundary. Let $r_t$ denote the time $t$ flow of the vector field $\psi R$. Then $r_1\circ f_0$ gives the required embedding. In other words, we just \emph{pushed} $f_0(V)$ into the interior of $\mathbb{D}^{2n-k}$ relative to its boundary.
		
		\item Let $g_1,g_2$ be two proper embeddings of $V$ in $\mathbb{D}^{n-k}$ such that $g_1$ and $g_2$ agree on a neighborhood of $\partial V$. We assume that $g_1$ and $g_2$ are transverse to the boundary $\partial D^{2n-k}$. Recall the embedding $f_0$ as in $(1)$ above. Then $\iota_1 = f_0 \cup g_1$ and $\iota_2 = f_0 \cup g_2$ are embeddings of the double, $D(V)$, of $V$. Let $v_1$ and $v_2$ be two corresponding normal vector fields on $\iota_1(D(V))$ and $\iota_2(D(V))$. Since $\iota_1 = \iota_2$ on a neighborhood $N(\partial V)$ of $\partial V \subset V$ and on the image of $f_0$, we may assume that $v_1 = v_2$ on $\mathcal{N} = N(\partial V) \cup f_0(V)$.
		
		If $g_1$ and $g_2$ are regularly homotopic, relative to $\mathcal{N}$ and with homotopic normal vector fields then by Theorem \ref{hh 3.1} $(\iota_1|_{D(V)_0},v_1)$ and $(\iota_1|_{D(V)_0},v_2)$ are isotopic. Moreover, this isotopy can be taken to be identity near $\mathcal{N}$. Therefore, we can proceed as in the proof of Theorem \ref{hh isotopy} to show that $\iota_1$ and $\iota_2$ are isotopic in $\mathbb{D}^{2n-k}$ via an isotopy $\Phi_t$. Then, $\Phi_t|_V$ gives us the required isotopy between $g_1$ and $g_2$.
		
		The surjectivity case can also be dealt with similarly. For details see \cite{HH}.
		
		\item We observe what happens to the proof of Theorem \ref{hh isotopy} when we look at an embedding $f$ of $M^n$ in $\mathbb{R}^{2n-k+1}$, instead of $\mathbb{R}^{2n-k}$. Let $Y = \mathbb{R}^{2n-k+1} \setminus f(M)$. Note that $Y$ is then $n$-connected and the homotopy classes of normal vector fields on $M_0$ are in $1-1$ correspondence with $H^{n-k}(M_0) = 0$. Let $f_1,f_2$ be two such embeddings. By Smale-Hirsch immersion theory the obstruction to a regular homotopy between $(f_1|_{M_0},v_1)$ and $(f_2|_{M_0},v_2)$ lies in the group $H^i(M_0;\pi_iSt(2n-k+1,n+1))$ for $1\leq i \leq n$. Since $M_0$ is homotopic to its $(n-k-1)$-skeleton and $St(2n-k+1,n+1)$ is $(n-k-1)$-connected, these obstructions vanish. Therefore, $f_1$ and $f_2$ are isotopic on $M_0$. Since $H_n(Y) = 0$, both $f_{v_1}(M)$ and $f_{v_2}(M)$ are null-homologous in $Y$. Arguing as in Theorem \ref{hh isotopy}, we then see that $f_1|_{D_1}$ and $f_2|_{D_1}$ are homotopic relative to boundary. Thus, by Theorem \ref{hh 3.2} we see that $f_1$ and $f_2$ are isotopic. Using arguments similar to $(1)$ and $(2)$ above, one can then show that any two proper embeddings of $(V^n,\partial V)$ in $(\mathbb{D}^{2n-k+1},\partial \mathbb{D}^{2n-k+1})$ are isotopic.
	\end{enumerate}
	
\end{proof}

\begin{proof}[\textbf{Proof of Theorem \ref{open book thm 1}}]
	Say, $M^n = \mathcal{O}b(V^{n-1},\phi)$, where the open book is obtained by the Winkelnkemper construction. Since $M$ is $k$-connected, by Lemma \ref{le wink}, $V^{n-1}$ is $k$-connected. Using $(1)$ of Lemma \ref{relative HH isotopy}, take a proper embedding $f : (V^{n-1},\partial V^{n-1}) \rightarrow (\mathbb{D}^{2n-k-1},\partial \mathbb{D}^{2n-k-1})$. By $(3)$ of Lemma \ref{relative HH isotopy}, $f$ and $f \circ \phi$ are isotopic (relative to boundary) embeddings of $(V^{n-1},\partial V^{n-1})$ in $(\mathbb{D}^{2n-k-1},\partial \mathbb{D}^{2n-k-1})$. By the isotopy extension theorem, there exists an isotopy  $\Phi_t$ of $(\mathbb{D}^{2n-k-1},\partial \mathbb{D}^{2n-k-1})$ for $t \in [0,1]$, such that $\Phi_0 = id$ and $\Phi_1 \circ f = f \circ \phi$. Moreover, $\Phi_t$ can be arranged to be identity near boundary. Hence, $\mathcal{O}b(V^{n-1},\phi)$ open book embeds in $\mathcal{O}b(\mathbb{D}^{2n-k-1},\Phi_1) \cong \mathcal{O}b(\mathbb{D}^{2n-k-1}, id) \cong \mathbb{S}^{2n-k}$.  
\end{proof}	

\begin{remark}
	Note that for manifolds which admit open book decomposition with identity monodromy, one can get open book embedding in $S^{2n-k-1} \cong \mathcal{O}b(D^{2n-2k-2}, id)$. If $N^{n-1}$ is a manifold with boundary, then $\partial(N \times D^2)$ admits such an open book with page $N^{n-1}$.
\end{remark}

\section{Smooth open book embedding of manifolds bounding an abstract Weinstein achiral fibration}\label{stein sect}

   In this section we prove Theorem \ref{top sein embed}. We prove it for the case when the achiral Lefschetz fibration has monodromy consisting of only positive Dehn--Seidel twist. The general case is completely analogous. The following lemma is our main ingredient.
   
   \begin{lemma}\label{top sein lemma}
   	Let $L \subset (W^{2n},d\lambda_W)$ be a Lagrangian $n$-sphere and let $\tau_L$ denote the Dehn-Seidel twist along $L$. Then, $(W^{2n};\tau_L)$ admits an embedding $g$ in $\mathbb{D}^{2\lfloor \frac{3n}{2} \rfloor + 2}$ such that there exists an isotopy $\psi_t$ ($t \in [0,1]$) of $\mathbb{D}^{2\lfloor \frac{3n}{2} \rfloor + 2}$, relative to boundary, whose time $1$-map induces the map $\tau_L$ on the embedded $W^{2n}$ and $g(W^{2n}) = \psi_1 \circ g(W^{2n}).$ 
  
   \end{lemma}
   
   Before going to the proof of Lemma \ref{top sein lemma}, let us recall a classical result from embedding theory.
   
   \begin{lemma} \label{embedding of n-handlebody}
   Let $V^{2m}$ be a compact manifold with nonempty boundary which has handles of index upto $m$. Then $V^{2m}$ admits a smooth embedding in $\mathbb{R}^{2m+1}.$ 	
   \end{lemma}

\begin{proof}
	
Since $V^{2m}$ has the homotopy type of an $m$-dimensional CW-complex, by the Smale-Hirsch immersion theory $V^{2m}$ admits an immersion $f_0$ in $\mathbb{R}^{2m+1}.$ By Whitney (Theorem $2(e)$, \cite{Wh2}, $f_0$ can be approximated by an immersion $f$	that restricts to an embedding on the $m$-skeleton $V^{(m)}$ of $V^{2m}$. By a compactness argument one can then prove that a slightly thickened neighborhood of $V^{(m)}$ in $V^{2m}$, which is again diffeomorphic to $V^{2m}$, can be embedded in $\mathbb{R}^{2m+1}$ (see the proof of Theorem $4.1$ in \cite{Hi2}).
	
\end{proof}

   \begin{proof}[Proof of Lemma \ref{top sein lemma}]
   	By Lemma \ref{embedding of n-handlebody}, $W^{2n}$ admits an embedding $f_0$ in $\mathbb{S}^{2\lfloor\frac{3n}{2}\rfloor + 1}$. Let $N = 2\lfloor\frac{3n}{2}\rfloor + 1$. 
   	
   	By the Lagrangian neighborhood Theorem, $L \subset (W^{2n},d\lambda_W)$ has a neighborhood $N(L)$ symplectomorphic to $(DT^*\mathbb{S}^n,d\lambda_n)$. This induces a smooth embedding $h_0 = f_0|_{N(L)}$ of $(DT^*\mathbb{S}^n,d\lambda_n)$ in $\mathbb{S}^N$. Recall that the standard inclusion of $\mathbb{S}^n$ in $\mathbb{S}^{n+k}$ induces a standard proper isosymplectic embedding of $(DT^*\mathbb{S}^n,d\lambda_n)$ in $(DT^*\mathbb{S}^{n+k},d\lambda_{n+k})$ for all $k \geq 0$. This gives an embedding of $DT^*\mathbb{S}^n$ in $\mathcal{O}b(DT^*\mathbb{S}^{n+k}, \tau_{n+k}) = \mathbb{S}^{2n+2k+1}$ such that an $(n+k)$-fold Dehn--Seidel twist on $DT^*\mathbb{S}^{n+k}$ induces an $n$-fold Dehn--Seidel twist on the embedded $DT^*\mathbb{S}^n$. Note that there exists a flow $\Phi_t$ on $\mathbb{S}^{2n+2k+1}$ whose time $1$ map $\Phi_1$ maps its page $DT^*\mathbb{S}^{n+k}$ to itself by applying a Dehn--Seidel twist on $DT^*\mathbb{S}^{n+k}$. Thus, we get another embedding $h_1$ of $DT^*\mathbb{S}^n$ in $\mathbb{S}^{N}.$
   	
   	Consider the $(N+1)$-disk $\mathbb{D}^{N+1}$ and let $\mathbb{S}^N \times [-1, 2)$ denote a collar neighborhood of $\partial \mathbb{D}^{N+1}$ such that $\mathbb{S}^{N} \times \{-1\}$ denotes $\partial \mathbb{D}^{N+1}$. Note that one can think of $W^{2n}$ as embedded in $\mathbb{S}^N \times \{0\}$ and then attach $\partial W^{2n} \times [-1,0] \subset \mathbb{S}^N \times [-1,0]$ to $W^{2n}$ along boundary to make $f_0$ a proper embedding of $W$ in $\mathbb{D}^{N+1}$.

   	Haefliger--Hirsch \cite{HH} generalized the result of Wu \cite{Wu} to show that any two embeddings of a $k$-connected manifold $M^n$ ($k \leq \frac{n-1}{2}$) in $\mathbb{S}^p$ are isotopic, for $p > \frac{3n}{2} + 2$. Therefore, $h_0$ and $h_1$ are isotopic in $\mathbb{S}^N$. Let $h_t$ be an isotopy between $h_0$ and $h_1$ in $\mathbb{S}^N$ for $t \in [0,1]$. We now modify the proper embedding $f_0$ in the following way. Take $\epsilon \in (0, \frac{1}{2}).$ Define $H : \partial DT^*\mathbb{S}^n \times [0,1] \rightarrow \mathbb{S}^N \times [0,1-2\epsilon]$ by $H(x,s) = (h_{s(1-2\epsilon)}|_{\partial DT^*\mathbb{S}^n}(x),s).$  We remove the interior of $f_0(N(L))$ from $f_0(W)$ and add the cylindrical region $H(\partial DT^*\mathbb{S}^n \times [0,1-2\epsilon])$ to $\overline{f_0(W)\setminus f_0(N(L))}$ along its boundary. We then further glue $h_1(\partial DT^*\mathbb{S}^n) \times [1-2\epsilon,1] \subset \mathbb{S}^N \times [1-2\epsilon,1]$ along $\partial(h_1(DT^*\mathbb{S}^n) \times \{1-2\epsilon\}).$ Finally, we identify the boundary of $h_1(DT^*\mathbb{S}^n) \times \{1\}$ with $h_1(\partial DT^*\mathbb{S}^n) \times \{1\}$. The total space, $\overline{f_0(W)\setminus f_0(N(L))} \cup H(\partial DT^*\mathbb{S}^n \times [0,1-2\epsilon]) \cup (h_1(\partial DT^*\mathbb{S}^n) \times [1-2\epsilon,1]) \cup h_1(DT^*\mathbb{S}^n) \times \{1\}$, gives a new proper embedding $f_1$ of $W^{2n}$ in $\mathbb{D}^{N+1}$.

   	\noindent Thus, we now have two proper embeddings, $f_0$ and $f_1$, of $W^{2n}$ inside $\mathbb{D}^{N+1}$. Note that $f_0 = f_1$ in the complement of $N(L)$. Thus, by applying Lemma \ref{relative HH isotopy}, we can find an isotopy between $f_0$ and $f_1$ supported away from the complement of $N(L)$. We extend this isotopy to an isotopy $G_s$ of $\mathbb{D}^{N+1}$, relative to boundary, such that $G_0 = id$ and $G_1 \circ f_0 = f_1.$

   	Let $\Psi_t$ be the isotopy of $\mathbb{S}^N$ such that $\Psi_0 = id$ and $\Psi_1$ realizes
   	the Dehn-Seidel twist on the page of its open book $\mathcal{O}b(DT^*\mathbb{S}^{\frac{N-1}{2}}, \tau_{\frac{N-1}{2}}).$  We construct an isotopy $\Gamma_s$ ($0\leq s \leq 1$) of $\mathbb{S}^N \times [-1,2]$ that satisfies the following properties.
   	
   	\begin{enumerate}
   		
   		\item $\Gamma_0 = id.$ 
   		
   		\item $\Gamma_1$ restricted to $\mathbb{S}^N \times \{1\}$ is $\Psi_1.$
   		
   	\end{enumerate}

   	\noindent $\Gamma_s$ is defined as follows.
   	
   	$$
   	\Gamma_s (x,t) 
   	=
   	\left\{
   	\begin{array}{ll}
   	\Psi_{s(2-t)}(x)  & \mbox{if } t \geq 1 \\
   	\Psi_{s(\frac{1+t}{2})}(x) & \mbox{if } t \leq 1
   	\end{array}
   	\right . 
   	$$
   	
   	\begin{figure}[htbp] 
   		
   		\def\svgwidth{16cm}
   		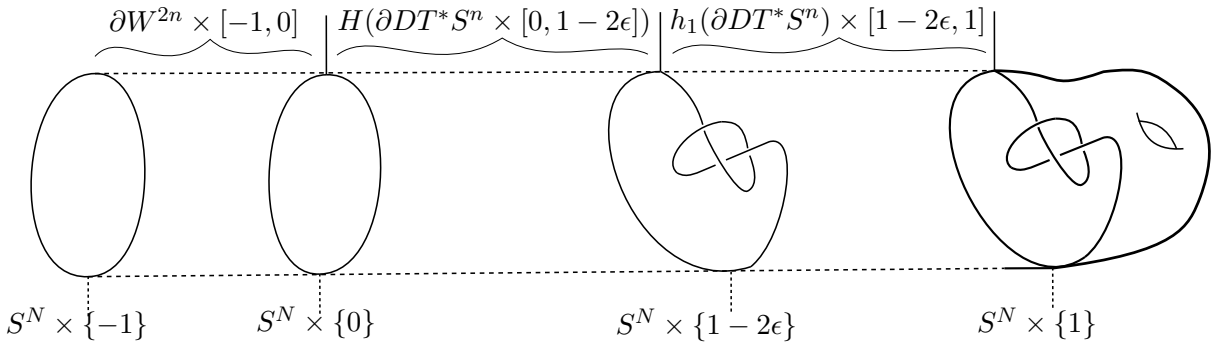
   		
   		\caption{Embedding $W^{2n}$ in $\mathbb{S}^N \times [-1,2]$ via $f_1.$}
   		\label{tsn}
   	\end{figure}

   	\noindent Let $\alpha^{-1}: [1-\epsilon, 1 + \epsilon] \rightarrow [-1,2]$ be a linear diffeomorphism such that $\alpha(1) = 1$. Using $\alpha^{-1}$, we can get an isotopy $\tilde{\Gamma}_s$ of $\mathbb{D}^{N+1}$, supported on $\mathbb{S}^N \times [1-\epsilon, 1 + \epsilon] \subset \mathbb{S}^{N} \times [-1,2] \subset \mathbb{D}^{N+1}$, so that $\tilde{\Gamma}_0 = id$ and $\tilde{\Gamma}_1|_{\mathbb{S}^N \times \{1\}} = \Psi_1.$

   	We now piece together all the steps to finish the proof. First we embed $W^{2n}$ in $\mathbb{S}^N \times \{0\}$ via $f_0$. Then we apply $G_s$ to isotope the neighborhood $N(L)$ into the page of the standard open book of $\mathbb{S}^N \times \{1\}$. We then apply $\tilde{\Gamma}_s$ to induce a Dehn-Seidel twist on $G_1(N(L))$. By construction, $\tilde{\Gamma}_s$ restricts to the identity map in the complement of $G_1(N(L)) \subset G_1(W)$. Finally, we bring $G_1(N(L))$ back to its starting embedding in $\mathbb{S}^N \times \{0\}$ by the isotopy $G_{1-s}$. The resultant diffeomorphism on $\mathbb{D}^{N+1}$ is isotopic, relative to boundary, to the identity map of $\mathbb{D}^{N+1}$, and induces a Dehn-Seidel twist on $N(L) \subset f_0(W)$. 
   	
   \end{proof}

\begin{proof}[Proof of Theorem \ref{top sein embed}]
	Say $M^{2n+1}$ bounds an achiral Lefschetz fibration $V^{2n+2} = LF(W^{2n},\phi)$ with monodromy consisting of Dehn--Seidel twists along the Lagrangians $(L_1,L_2,\dots,L_m))$. This induces an open book decomposition on $M^{2n+1}$ with page $W$ and monodromy $\phi = \tau_{L_1}\circ \tau_{L_2}\circ\dots\circ \tau_{L_m}$. Identify the $\mathbb{S}^1$-interval of the mapping torus of this open book with $[0,1]/0 \sim 1$. Divide $[0,1]$ into sub-intervals $I_j = [\frac{j}{m+1},\frac{j+1}{m+1}]$, for $j \in \{0,1,...,m\}$. Let $N = 2\lfloor\frac{3n}{2}\rfloor + 1$ as before. We apply Lemma \ref{top sein lemma} on the interval $I_j$, for the vanishing cycle $L_j$, to obtain embedding of the mapping cylinder of $\tau_{L_j} : W \rightarrow W$ in the mapping cylinder of the identity map (up to isotopy) of $\mathbb{D}^{N+1}$, for $0 \leq j \leq m-1$. We then glue these mapping cylinders together in cyclic order to obtain embedding of the mapping torus of $\phi$ in $\mathbb{D}^{N+1} \times \mathbb{S}^1$. This gives us an open book embedding of $M^{2n+1} = \mathcal{O}b(W^{2n},\phi)$ in $\mathbb{S}^{N+2} = \mathcal{O}b(\mathbb{D}^{N+1},id)$.

\end{proof}

\section{Contact embedding of manifolds bounding an achiral Lefschetz fibration}

For isocontact embeddings, Gromov (\cite{Gr}) proved the following theorem. The statement here is taken from \cite{EM}.

\begin{theorem}[Gromov] \label{h-principle for contact embedding}
	Let $(M,\xi)$ and $(V,\eta)$ be contact manifolds of dimension $2n+1$ and $2N+1$ respectively. Suppose that the differential $F_0 = Df_0$ of an embedding $f_0 : (M,\xi) \rightarrow (V,\eta)$ is homotopic via a homotopy of monomorphisms $F_t : TM \rightarrow TV$ covering $f_0$ to a contact monomorphism $F_1 : TM \rightarrow TV$. The embedding $f_0$ is called a \emph{formal contact embedding}.
	
	\begin{enumerate}
		\item 	Open case: If $n \leq N-1$ and the manifold $M$ is open, then there exists an isotopy $f_t : M \rightarrow V$ such that the embedding $f_1 : M \rightarrow W$ is contact and the differential $Df_1$ is homotopic to $F_1$ through contact monomorphisms.
		\item 	Closed case: If $n \leq N-2$, then the above isotopy $f_t$ exists even if $M$ is closed.
	\end{enumerate}
	
\end{theorem}

Note that Theorem \ref{h-principle for contact embedding} implies that a smooth embedding of a contact manifold, which is regularly homotopic to a contact immersion, is isotopic to an isocontact embedding.

The following result was proved in \cite{S2}.

\begin{lemma}\label{lemma S2}
	If an almost contact manifold $M^{2n+1}$ embeds in ${\mathbb{R}}^{2N+1}$ with a trivial normal bundle, then there exists a contact structure $\xi_0$ such that $(M,\xi_0)$ isocontact embeds into $({\mathbb{R}}^{2N+1},\eta)$ ($N-n\geq1$) for some contact structure $\eta$ on $\mathbb{R}^{2N+1}$. 
\end{lemma}

The idea of the proof of Lemma \ref{lemma S2} is to put a contact structure on $M^{2n+1} \times \mathbb{D}^{2(N-n)} \subset \mathbb{R}^{2N+1}$ and then look at the obstructions to extend the contact structure as an almost contact structure to all of $\mathbb{R}^{2N+1}.$ The obstructions to such extension lies in the groups $H^{i+1}(\mathbb{R}^{2N+1},M;\pi_i(\Gamma_{N+1})) \cong H^i(M;\pi_i(\Gamma_{N+1}))$, for $1 \leq i \leq 2n+1$, where $\Gamma_{N+1}$ denotes the set of all complex structures on $\mathbb{R}^{2N+2}$. If all the obstructions vanish, then Theorem \ref{existence of contact structure gromov} gives the desired contact structure $\eta$ on $\mathbb{R}^{2N+1}$. Since $({\mathbb{R}}^{2N+1},\eta)$ isocontact embeds in $({\mathbb{R}}^{2N+3},\xi_{std}).$, the result follows.

Recently, a codimension--$2$ $h$-principle for isocontact embedding was proved for closed contact manifolds by Casals--Pancholi--Presas \cite{CPP}. In particular, they proved statement $(2)$ of Theorem \ref{h-principle for contact embedding} for the case $n \leq N - 1.$ This allows to improve the codimension of embedding in Lemma \ref{lemma S2} by $2$.

\begin{proof}[\textbf{Proof of Theorem \ref{codim 2 embedding}}]
	
	\begin{enumerate}
	
		\item Let $W^{2n+2}$ be an achiral Lefschetz fibration which has $M^{2n+1}$ as boundary. By Lemma \ref{embedding of n-handlebody}, $W^{2n+2}$ embeds in $\mathbb{R}^{2n+3}.$ In particular $M^{2n+1}$ embeds in $\mathbb{R}^{2n+3}.$ Recall that the normal bundle of every embedding of an orientable manifold in a Euclidean space has Euler class zero. Therefore, $M^{2n+1}$ embeds in $\mathbb{R}^{2n+3}$ with trivial normal bundle. By the proof of Lemma \ref{lemma S2}, there exists a contact structure $\xi_0$ on $M$ such that $(M,\xi_0)$ isocontact embeds in $(\mathbb{R}^{2n+3},\eta)$ for some contact structure $\eta$ on $\mathbb{R}^{2n+3}.$ Since any two almost contact structures on $\mathbb{R}^{2n+3} $are homotopic, $(M,\xi_0)$ admits a formal contact embedding in $(\mathbb{R}^{2n+3},\xi_{std}).$ By the codimension--$2$ $h$-principle of Casals--Pancholi--Presas \cite{CPP} for isocontact embeddings of closed contact manifolds, $(M,\xi)$ then isocontact embeds in $(\mathbb{R}^{2n+3},\xi_{std}).$
		
		\item Let $\xi$ be any contact structure on $M^{2n+1}$. Let $\alpha$ be a contact form representing $\xi$. Let $(r,\theta)$ be a cylindrical co-ordinate system on $\mathbb{D}^2$. The $1$-form $\tilde{\alpha} = \alpha + r^2 d\theta$ then defines a contact structure on $M \times \mathbb{D}^2 \subset \mathbb{R}^{2n+3}$. The obstructions to extend this as an almost contact structure to all of $\mathbb{R}^{2n+3}$ lies in $H^i(M;\pi_i(\Gamma_{n+2}))$, for $1 \leq i \leq 2n+1.$ Since $M^{2n+1}$ is $(n-1)$-connected, the only obstructions lie in $H^n(M;\pi_n(\Gamma_{n+2}))$, $H^{n+1}(M;\pi_{n+1}(\Gamma_{n+2})$ and $H^{2n+1}(M;\pi_{2n+1}(\Gamma_{n+2}))$. Since, $\pi_i(\Gamma_{n+2}) = \pi_{i+1}(SO)$ for $i \leq 2n+2$ and $\pi_{n+1}(SO) = \pi_{n+2}(SO) = \pi_{2n+2}(SO) = 0$ for $n \equiv 4 \pmod 8$, all the obstructions vanish. Thus, by Theorem \ref{existence of contact structure gromov}, $(M,\xi)$ admits an isocontact embedding in $(\mathbb{R}^{2n+3},\eta)$ for some contact structure $\eta$ on $\mathbb{R}^{2n+3}.$ Then by similar arguments as in $(1)$ above, one can show that $(M^{2n+1},\xi)$ isocontact embeds in $(\mathbb{R}^{2n+3},\xi_{std}).$ 
		
	\end{enumerate}

\end{proof}

\section{Contact open book embedding in $(S^{4n+1},\xi_{std})$}\label{contact ob sect}

\subsection{$h$-principle for isosymplectic immersions and isosymplectic embeddings} We recall the notions of \emph{isosymplectic immersions}, \emph{isosymplectic embeddings} and \emph{formal isosymplectic immersion} or \emph{isosymplectic homomorphism}.

\begin{definition}[Isosymplectic immersion/embedding]
	An isosymplectic immersion (or embedding) $h : (M^{2n}, \omega_M) \rightarrow (W^{2N},\omega_W)$ is an immersion (or embedding) such that $h^*(\omega_W) = \omega_M$.  
\end{definition}

\begin{definition}[Isosymplectic homomorphism]
	An isosymplectic homomorphism is a map $F : (TM^{2n},\omega_M) \rightarrow (TV^{2N},\omega_W)$ such that $F$ restricts to a vector space monomorphism on $(T_pM,\omega_M)$ for all $p \in M$ and $F^*(\omega_W) = \omega_M$.
\end{definition}

\noindent The existence of an isosymplectic homomorphism is a necessary condition for existence of an isosymplectic immersion. Gromov \cite{Gr} proved the following $h$-principle to show that it is also sufficient.

\begin{theorem}[$h$-principle for isosymplectic immersion \cite{Gr}] \label{gromov h-principle immersion}
	
	Let $(W, \omega_W)$ and $(V, \omega_V)$ be symplectic manifolds of dimensions $2N$ and $2n$ respectively. Suppose a
	continuous map $f_0 : V \rightarrow W$ satisfies the cohomological condition $f_0^*[\omega_W] = [\omega_V]$, and that $f_0$ is covered by an isosymplectic homomorphism $F : TV \rightarrow TW$. If $V$ is closed and $2N \geq 2n + 2$, then there exists a homotopy $f_t : V \rightarrow W$ such that $f_1$ is an isosymplectic immersion and the differential
	$Df_1$ is homotopic to $F$ through isosymplectic homomorphisms.
	
\end{theorem}

\noindent The $h$-principle for isosymplectic immersion is also true for the relative and the parametric version (see Theorem $16.4.3$ in \cite{EM}).

\

We now recall Gromov's $h$-principle for isosymplectic embedding.

\begin{theorem}[$h$-principle for isosymplectic embedding \cite{Gr}] \label{gromov h-principle embedding}
	
	Let $(V, \omega_V)$ and
	$(W, \omega_W)$ be symplectic manifolds of dimension $n = 2m$ and $q = 2l$ respectively. Suppose that an embedding $f_0 : V \rightarrow W$ satisfies the cohomological
	condition $f_0^*[\omega_W] = [\omega_V]$, and the differential $F_0 = df_0$ is homotopic via a
	homotopy of monomorphisms
	$F_t : TV \rightarrow TW$, covering $f_0$,
	to an isosymplectic homomorphism $F_1 : TV \rightarrow TW$.
	
	\begin{enumerate}
		\item (Open case) If $n \leq q-2$ and the manifold $V$ is open then there
		exists an isotopy $f_t : V \rightarrow W$ such that the embedding $f_1 : V \rightarrow W$
		is isosymplectic and the differential $df_1$ is homotopic to $F_1$ through
		isosymplectic homomorphisms.
		
		\item (Closed case) If $n\leq q-4$ then the above isotopy $f_t$ exists even if $V$
		is closed. Moreover, one can choose the isotopy $f_t$ to be arbitrarily
		$C^0$-close to $f_0$.
	\end{enumerate}
	
\end{theorem}

The $h$-principle for isosymplectic embedding also holds true for the relative and the parametric case. In particular, if a family of proper isosymplectic immersions $f_t : (V, \partial V, d\lambda_V) \rightarrow (W, \partial W, d\lambda_W)$ is regularly homotopic to a family of embeddings $e_t$, depending smoothly on $t$, then by the parametric relative version of Theorem \ref{gromov h-principle embedding}, $f_t$ can be isotoped to a family of proper isosymplectic embeddings $f_t^s$.


It is well known from the obstruction theory of fiber bundles that $(TV,\omega_V)$ has an isosymplectic homomorphism in $(TD^{2N},\omega_0)$ if and only if all the obstruction classes in $H^{i}(V, \partial V;\pi_{i-1}{St}^\mathbb{C}(N,n))$ vanish for $1 \leq i \leq 2n$. Here, ${St}^\mathbb{C}(N,n)$ denotes the complex Stiefel manifold. Moreover, Two isosymplectic homomorphisms are homotopic via isosymplectic homomorphisms if all the homotopy obstructions in $H^{i}(V, \partial V;\pi_i{St}^\mathbb{C}(N,n))$, for $1 \leq i \leq 2n$.

\

The following lemma is essentially Proposition $4$ in \cite{Au}, adapted to our setting. For a proof, we refer to \cite{S}.

\begin{lemma}\label{symplectic isotopy lemma}
	Let $(V,\partial V,d\lambda_V) $ and $(W,\partial W,d\lambda_W)$ be two exact symplectic manifolds with convex boundaries of dimension $2m$ and $2m+2s$ respectively. Let $\psi_t : (V,\partial V) \rightarrow (W,\partial W)$ be a family of proper isosymplectic embeddings. There exists a symplectic isotopy $\Psi_t$ of $(W,\partial W,d\lambda_W)$ such that $\Psi_0 = Id$ and $\Psi_1 \circ \psi_0(V) = \psi_1(V)$.
\end{lemma}

\begin{proof}[\textbf{Proof of Theorem \ref{contact open book thm}}]
	
	Let $f^s$ and $f^s \circ \phi$ be two proper isosymplectic embeddings of $(V^{2n},\partial V^{2n},d\lambda_V)$ in $(D^{4n},\partial D^{4n},d\lambda_0)$. Since $V^{2n}$ is simply connected, by Lemma \ref{relative HH isotopy}, there exists a family of embeddings $e_t$ for $t \in [0,1]$ such that $e_0 = f^s$ and $e_1 = f^s \circ \phi$. Since $St^{\mathbb{C}}(2n,n)$ is $2n$-connected, $De_0$ and $De_1$ are homotopic via a continuous family of isosymplectic homomorphisms $F_t$ for $t \in [0,1]$. Applying the parametric relative version of the $h$-principle for isosymplectic immersion (Theorem \ref{gromov h-principle immersion}), we get a family of isosymplectic immersions $f^s_t$ joining $e_0$ and $e_1$, such that there is continuous family of formal homotopy $H_t^r$ ($r \in [0,1]$) connecting $H_t^0 = Df^s_t$ and $H_t^1 = F_t$. Since each $f^s_t$ is regularly homotopic to the embedding $e_0$ and hence to $e_t$, there exists continuous family $G_t^r$ ($r \in [0,1]$) of formal homotopy between $ G_t^0 = De_t$ and $G_t^1 = Df^s_t$. Define $(H\circ G)_t^r$ as follows.

	$$
	(H\circ G)_t^r
	=
	\left\{
	\begin{array}{ll}
	G_t^{2r}  & \mbox{for } 0 \leq r \leq \frac{1}{2}  \\
	H_t^{2r-1} & \mbox{for } \frac{1}{2} \leq r \leq 1
	\end{array}
	\right . 
	$$

	Thus, $(H\circ G)_t^r$ gives a continuous family of formal homotopy between $De_t$ and $F_t$. Moreover, we can choose the formal homotopy to vary smoothly with $t$. So, we have a family of embeddings $e_t$ satisfying the cohomology conditions for symplectic forms and $F_t$ is a family of isosymplectic homomorphisms covering $e_t$ such that $De_t$ is formally homotopic to $F_t$ and the homotopy depends continuously on $t$. Thus, by the parametric relative version of Theorem \ref{gromov h-principle embedding}, we can isotope each $e_t$ to an isosymplectic embedding depending continuously on $t$. Therefore, we have shown that $e_0 = f^s$ and $e_1 = f^s \circ \phi$ are isotopic as isosymplectic embeddings. By Lemma \ref{symplectic isotopy lemma}, there is a symplectic isotopy $\Psi_t$ of $(D^{4n},\partial D^{4n},d\lambda_0)$ such that $\Psi_0 = Id$ and $\Psi_1 \circ f^s(V) = f^s \circ \phi (V)$. Hence, $\mathcal{O}b(V^{2n},d\lambda_V,\phi)$ admits a contact open book embedding in $\mathcal{O}b(D^{4n},d\lambda_0,id) = (S^{4n+1},\xi_{std})$.
	
\end{proof}

\section{Some applications of Theorem \ref{contact open book thm}}

Theorem \ref{contact open book thm} says that if one can find a simply connected Weinstein domain $(W^{2n},d\lambda_W)$ that admits proper symplectic embedding in $(D^{4n},d\lambda_0)$, then $\mathcal{O}b(W^{2n},d\lambda_W,\phi)$ contact open book embeds in $\mathcal{O}b(D^{4n},d\lambda_0,id)$. Therefore, to apply Theorem \ref{contact open book thm} we want to find a sizable class of simply connected $2n$-Weinstein domains that admit proper symplectic embedding in the standard $4n$-ball.

\begin{proof}[Proof of corollary \ref{corollary1}]
	
	By Cieliebak \cite{C}, any sub-critical Weinstein domain is split. Thus, $(W^{2n},d\lambda_W)$ is symplectomorphic to some $(W_0^{2n-2} \times D^2, d\lambda_{W_0} \bigoplus dx \wedge dy)$, where $(W_0^{2n_2}, d\lambda_{W_0})$ is a Weinstein domain. Now, by Gromov's h-principle for isosymplectic embedding, $(W_0^{2n-2}, d\lambda_{W_0})$ has a proper isosymplectic embedding in $(D^{4n-2}, d\lambda^{2n-1}_0)$. Therefore, $(W^{2n},d\lambda_W) = (W_0^{2n-2} \times D^2, d\lambda_{W_0} \bigoplus dx \wedge dy)$ has a proper isosymplectic embedding in $(D^{4n}, d\lambda^{2n}_0)$. Hence, by Theorem \ref{contact open book thm} $\mathcal{O}b(W^{2n},d\lambda_W,\phi_W)$ contact open book embeds in $\mathcal{O}b(D^{4n},d\lambda_0,id)$.
	
\end{proof}

By Gromov's $h$-principle for symplectic immersion, there exists a proper isosymplectic immersion $f^s$ of $(V^{2n}, d\lambda_V)$ in $(D^{4n}, d\lambda_0)$. Moreover, any such immersion will have only finitely many intersection points. By Whitney \cite{Wh}, this intersection number $I(f^s) = \frac{1}{2} \langle e_r(\nu(f^s)),[V]\rangle$. Here, $\nu(f^s)$ denotes the normal bundle of immersion, $e_r(\nu(f^s)) \in H^{2n}(V,\partial V ; \mathbb{Z})$ denotes the relative Euler class of $\nu(f^s)$, and $[V]$ is the fundamental class. Since an isosymplectic immersion induces a complex structure on the normal bundle, $e(\nu(f^s))$ is same as the relative Chern class $c^r_n(\nu(f^s))$. Therefore, $I(f^s) = \frac{1}{2} \langle c^r_n(\nu(f^s)),[V]\rangle$. Moreover, by Whitney \cite{Wh}, $f^s$ is regularly homotopic to an embedding if and only if $I(f^s) = 0$. Therefore, by Theorem \ref{gromov h-principle embedding} we have the following.

\begin{lemma}\label{intersection number}
	Let $g : (V^{2n}, d\lambda_V) \rightarrow (\mathbb{D}^{4n}, d\lambda_0)$ be a proper isosymplectic immersion with double points in the interior of $V^{2n}$. Then, the intersection number $ I(g) = \frac{1}{2} \langle c^r_n(\nu(g)),[V]\rangle$. Moreover, $I(g) = 0$ if and only if $g$ can be isotoped to an isosymplectic embedding.  
\end{lemma}

Recall that if $(E,F,\pi)$ is an $n$-dimensional complex vector bundle, then its $n^{th}$-Segre class $s_n(E)$ is the $n^{th}$ order term in the expansion of $(1 + c_1(E) + c_2(E) +...+ c_n(E))^{-1}$. We then show the following.

\begin{lemma} \label{symp embed lemma}
	
	Let $(V^{2n},d\lambda_V)$ be a Weinstein domain with a compatible almost complex structure $J_V$. Assume that $(V^{2n},d\lambda_V)$ satisfies the following properties.
	
	\begin{enumerate}
		
		\item The restriction of $(TV,J_V)$ on $\partial V$ is trivial as a complex vector bundle.
		
		\item The $n^{th}$ relative Segre class of $(V^{2n},\partial V)$, $s^r_n(V,\partial V)$, vanishes.
		
	\end{enumerate}
	
	Then $(V^{2n},d\lambda_V)$ admits a proper isosymplectic embedding in $(\mathbb{D}^{4n},d\lambda_0)$, for $n \geq 3$.
	
\end{lemma}	

The proof is an application of a version of the Whitney sum formula for relative characteristic classes. It was first developed by Kervaire in \cite{Ker}.

\begin{proof}[Proof of Lemma \ref{symp embed lemma}:]

	Let $f_s$ be a proper isosymplectic immersion of $(V^{2n},d\lambda_V)$ in $(\mathbb{D}^{4n},d\lambda_0)$. Note that the $i$-th Chern class is a characteristic class in the group $H^{2i}(V^{2n}; \pi_{2i-1}St^\mathbb{C}(n,n-i+1))$, which is an obstruction to the existence of a section of the associated $St^\mathbb{C}(n,n-i+1)$-bundle of $(TV,J_V)$. Since the restriction of $(TV,J_V)$ on $\partial V$ is a trivial complex bundle, their is no obstruction to the existence of a section of that associated $St^\mathbb{C}(n,n-i+1)$-bundle over $\partial V$. Therefore, we can talk about relative Chern classes of $(TV,J_V)$. Moreover, all the relative Chern classes of $(T\mathbb{D}^{4n},d\lambda_0,J_0)$ vanish. Therefore, the relative version of the Whitney sum formula gives the following.  $$(1 + c^r_1(V,\partial V) + \dots + c^r_n(V, \partial V))(1 + c^r_1(\nu(f_s)) + \dots + c^r_n(\nu(f_s))) = 1 \implies c^r_n(\nu(f_s)) = s^r_n(V,\partial V) = 0.$$
	
	\noindent Thus, by Lemma \ref{intersection number}, $f_s$ can be isotoped to a proper isosymplectic embedding of $(V,d\lambda_V)$ in $(\mathbb{D}^{4n},d\lambda_0).$
	
\end{proof}

\begin{proof}[Proof of Corollary \ref{corollary2}]
	
	Recall that a Weinstein domain has the homotopy type of an $n$-complex. The obstruction to finding a complex trivializtion of $(TW,d\lambda_W,J_W)$ lies in $H^i(W^{2n};\pi_{i-1}U(n)) = H^i(W;\pi_{i-1}U)$ ($n \geq 3 \implies n < 2n-2$), for $1 \leq i \leq n$. By Bott periodicity, $\pi_kU = 0$ for $k$ even and there are no nonzero even cells. Thus, $(TW,d\lambda_W,J_W)$ is trivial and the conditions of Lemma \ref{symp embed lemma} are satisfied. Thus, $(W^{2n},d\lambda_W)$ admits proper isosymplectic embedding in the standard symplectic $4n$-ball. Hence, by Theorem \ref{contact open book thm}, $\mathcal{O}b(W^{2n},d\lambda_W,\phi)$ open book embeds in $\mathcal{O}b(\mathbb{D}^{4n},d\lambda_0, id)$ for every relative symplectomorphism $\phi$ of $(W,d\lambda_W)$.  
	
\end{proof}

\begin{remark}
	In the proof of Corollary \ref{corollary2}, the main point was to show the triviality of the almost complex structure on the page. Therefore, similar result will hold true for any complex parallelizable Weinstein domain as page.
\end{remark}

\end{document}